\documentclass[fullpage, 11pt]{article}
\usepackage{amsfonts}
\usepackage{amssymb}
\usepackage{amsmath}
\usepackage{latexsym}
\usepackage{amsthm}
\usepackage{epsfig}
\usepackage{graphicx}
\usepackage{subfig}
\usepackage{color}

\newcommand{\old}[1]{{}}

\newtheorem{theorem}{Theorem}[section]
\newtheorem{corollary}[theorem]{Corollary}
\newtheorem{lemma}[theorem]{Lemma}

\newtheorem{remark}[theorem]{Remark}
\newtheorem{observation}[theorem]{Observation}
\newtheorem{claim}[theorem]{Claim}

\newtheorem{example}[theorem]{Example}

\newtheorem{assumption}[theorem]{Assumption}

\begin{document}

\title{On the Relative Strength of \\ Split, Triangle and Quadrilateral Cuts}

\author{
Amitabh Basu \thanks{Supported by a Mellon Fellowship.}   \\
Tepper School of Business, Carnegie Mellon University, Pittsburgh,
PA 15213
\\ abasu1@andrew.cmu.edu
\\ $\;$ \\
Pierre Bonami \thanks{Supported by ANR grant BLAN06-1-138894.}   \\
LIF, Facult\'e des Sciences de Luminy, Universit\'e de Marseille,
France
\\ pierre.bonami@lif.univ-mrs.fr
\\ $\;$ \\
G\'erard Cornu\'ejols \thanks{Supported by  NSF grant CMMI0653419,
ONR grant N00014-97-1-0196 and ANR grant BLAN06-1-138894.}   \\
Tepper School of Business, Carnegie Mellon University, Pittsburgh,
PA 15213 \\ and LIF, Facult\'e des Sciences de Luminy, Universit\'e
de Marseille, France
\\ gc0v@andrew.cmu.edu
\\ $\;$ \\
Fran\c{c}ois Margot \thanks{Supported by  ONR grant N00014-97-1-0196.}\\
Tepper School of Business, Carnegie Mellon University, Pittsburgh,
PA 15213 \\ fmargot@andrew.cmu.edu}

\date{July 2008, revised December 2008}

\maketitle

\begin{abstract}
Integer programs defined by two equations with two free integer
variables and nonnegative continuous variables have three types of
nontrivial facets: split, triangle or quadrilateral inequalities. In this
paper, we compare the strength of these three families of inequalities.
In particular we study how  well each family approximates the integer hull.
We show that, in a well defined sense, triangle inequalities provide
a good approximation of the integer hull. The same statement holds for
quadrilateral inequalities. On the other hand, the approximation produced by
split inequalities may be arbitrarily bad.
\end{abstract}

\section{Introduction}\label{Sec:Introduction}

In this paper, we consider mixed integer linear programs with two
equality constraints, two free integer variables and any number of
nonnegative continuous variables. We assume that the two integer
variables are expressed in terms of the remaining variables as
follows.
\begin{equation} \label{SI}
\begin{array}{rrcl}
          & x & = & f + \sum_{j=1}^{k} r^j s_j        \\
           &  x  & \in & \mathbb{Z}^2                  \\
          &  s  & \in & \mathbb{R}_+^k .
\end{array}
\end{equation}

This model was introduced by Andersen, Louveaux, Weismantel and Wolsey \cite{alww}. It is a natural relaxation of a general mixed integer linear
program (MILP) and therefore it can be used to generate cutting
planes for MILP. Currently, MILP solvers rely on cuts that
can be generated from a single equation  (such as Gomory mixed integer cuts
\cite{go}, MIR cuts \cite{mw}, lift-and-project cuts \cite{bcc},
lifted cover inequalities \cite{cjp}). Model (\ref{SI}) has attracted
attention recently as a way of generating new families of cuts from two
equations instead of just a single one \cite{alww,cm,dw,esp, go07}.

We assume $f \in \mathbb{Q}^2 \setminus \mathbb{Z}^2$, $k \ge 1$,
and $r^j \in \mathbb{Q}^2 \setminus \left\{ 0 \right\}$. So $s=0$ is
not a solution of (\ref{SI}).  Let $R_f(r^1,\ldots,r^k)$ be the
convex hull of all vectors $s \in \mathbb{R}_+^k$ such that $f +
\sum_{j=1}^{k} r^j s_j$ is integral. A classical theorem of Meyer
\cite{meyer} implies that $R_f(r^1,\ldots,r^k)$ is a polyhedron.
Andersen, Louveaux, Weismantel and Wolsey \cite{alww} showed that
the facets of $R_f(r^1,\ldots,r^k)$ are $s \geq 0$ (called {\it
trivial inequalities}), split inequalities \cite{COOK} and
intersection cuts (Balas \cite{bal}) arising from triangles or
quadrilaterals in $\mathbb{R}^2$. Borozan and Cornu\'ejols \cite{bc}
investigated a relaxation of (\ref{SI}) where the vector $(s_1,
\ldots, s_k) \in \mathbb{R}_+^k$ is extended to infinite dimensions
by defining it for all directions $r^j \in \mathbb{Q}^2$ instead of
just $r^1, \ldots , r^k$. Namely let $R_f$ be the convex hull of all
infinite-dimensional vectors $s$ with finite support that satisfy
\begin{equation} \label{SIinfinite}
\begin{array}{rrcl}
          & x & = & f + \sum_{r \in \mathbb{Q}^2} r s_r        \\
           &  x  & \in & \mathbb{Z}^2                  \\
          &  s  & \geq  & 0.
\end{array}
\end{equation}
Theorem~\ref{THM:minfunc} below shows that there is a one-to-one
correspondance between minimal valid inequalities for $R_f$ and
maximal lattice-free convex sets that contain $f$ in their interior.
By {\em lattice-free convex set} we mean a convex set with
no integral point in its interior. However integral points are
allowed on the boundary. These maximal lattice-free convex sets
are splits, triangles, and quadrilaterals
as proved in the following theorem of Lov\'asz \cite{lovasz}.

\begin{theorem} \label{THM:maxconv}
{\em (Lov\'asz \cite{lovasz})}
In the plane,
a maximal lattice-free convex set with nonempty interior
is one of the following:
\begin{itemize}
\item[(i)] A split $c \le a x_1 + b x_2 \le c+1$ where $a$ and $b$ are coprime
integers and $c$ is an integer;
\item[(ii)] A triangle with an integral point in the interior of each of its edges;
\item[(iii)] A quadrilateral containing exactly four integral points,
with exactly one of them in the interior of each of its edges;
Moreover, these four integral points are vertices of a parallelogram
of area 1.
\end{itemize}
\end{theorem}

$R_f(r^1,\ldots,r^k)$ is a polyhedron of blocking type, i.e.
$R_f(r^1,\ldots,r^k) \subseteq \mathbb{R}^k_+$ and if $x \in
R_f(r^1,\ldots,r^k)$, then $y \geq x$ implies $y \in
R_f(r^1,\ldots,r^k)$. Similarly $R_f$ is a convex set of blocking
type. Any nontrivial valid linear inequality for $R_f$ is of the
form
\begin{equation} \label{valid:infinite}
\sum_{r \in \mathbb{Q}^2} \psi(r) s_r \geq 1
\end{equation}
where $\psi: \mathbb{Q}^2 \rightarrow \mathbb{R}$. The nontrivial
valid linear inequalities for $R_f(r^1,\ldots,r^k)$ are the
restrictions of (\ref{valid:infinite}) to $r^1, \ldots , r^k$
\cite{cm}:
\begin{equation} \label{valid}
\sum_{j=1}^k \psi(r^j) s_j \geq 1.
\end{equation}
A nontrivial valid linear inequality for $R_f$ is {\em minimal} if
there is no other nontrivial valid inequality $\sum_{r \in
\mathbb{Q}^2} \psi'(r) s_r \geq 1$ such that $\psi'(r) \leq \psi(r)$
for all $r \in \mathbb{Q}^2$.

\begin{theorem} \label{THM:minfunc}
{\em (Borozan and Cornu\'ejols \cite{bc})} Minimal nontrivial valid
linear inequalities for $R_f$ are associated with functions $\psi$
that are nonnegative positively homogeneous piecewise linear and
convex. Furthermore, the closure of the set
\begin{equation} \label{Bpsi}
B_{\psi} := \{ x \in \mathbb{Q}^2: \; \psi(x-f) \leq 1 \}
\end{equation}
is a maximal lattice-free convex set containing $f$ in its interior.
\end{theorem}

Conversely, any maximal lattice-free convex set $B$ with $f$ in its
interior defines a function $\psi_B: \mathbb{R}^2 \rightarrow
\mathbb{R}_+$ that can be used to generate a minimal nontrivial
valid linear inequality. Indeed, define $\psi_B(0) = 0$ and
$\psi_B(x-f)=1$ for all points $x$ on the boundary of $B$. Then, the
positive homogeneity of $\psi_B$ implies the value of $\psi_B(r)$
for any vector $r \in \mathbb{R}^2 \setminus \{ 0 \}$: If there is a
positive scalar $\lambda$ such that the point $f + \lambda r$ is on
the boundary of $B$, we get that $\psi_B(r) = 1/\lambda$. Otherwise,
if there is no such $\lambda$, $r$ is an unbounded direction of $B$
and $\psi_B(r) = 0$.

Note that the above construction of $\psi_B$ is nothing but the derivation of the intersection cut as introduced by Balas \cite{bal}.

Following Dey and Wolsey \cite{dw}, the maximal lattice-free
triangles can be partitioned into three types (see Figure
\ref{The4cases}):

\begin{itemize}
\item {\it Type 1 triangles}: triangles with integral vertices
and exactly one integral  point in the relative interior of each edge;
\item {\it Type 2 triangles}: triangles with at least one fractional
vertex $v$, exactly one integral point in the relative interior
of the two edges incident to $v$ and at least two integral points
on the third edge;
\item {\it Type 3 triangles}: triangles with exactly three integral
points on the boundary, one in the relative interior
of each edge.
\end{itemize}
Figure \ref{The4cases} shows these three types of triangles as well as
a maximal lattice-free quadrilateral and a split satisfying the properties of
Theorem \ref{THM:maxconv}.

\begin{figure}[htbp]
\centering \scalebox{0.6}{ \epsfig{file=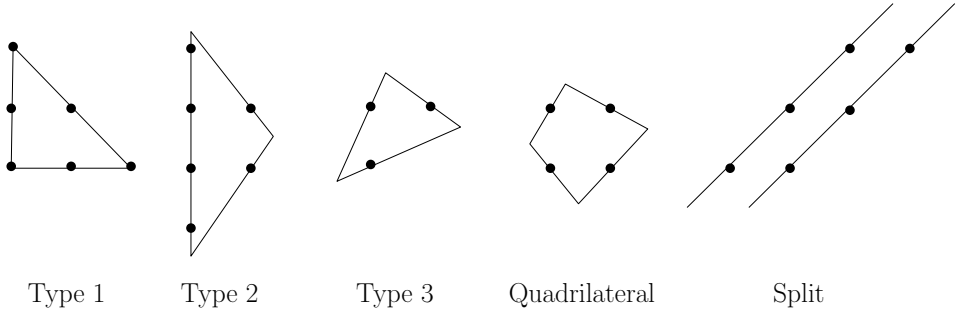} }
\caption{Maximal lattice-free convex sets with nonempty interior in
$\mathbb{R}^2$} \label{The4cases}
\end{figure}

In this paper we will need conditions guaranteeing that a split,
triangle or quadrilateral actually defines a facet of
$R_f(r^1,\ldots,r^k)$. Such conditions were obtained by Cornu\'ejols
and Margot \cite{cm} and will be stated in Theorem
\ref{THM:mainRfk}.

\subsection{Motivation}
\label{SUB:MOTIVATION}

An unbounded maximal lattice-free convex set is called a
{\it split}. It has two parallel edges whose direction is called the
{\it direction} of the split. Split inequalities for (\ref{SI}) are
valid inequalities that can be derived by combining the two
equations in (\ref{SI}) and by using the integrality of $\pi_1 x_1 +
\pi_2 x_2$, where $\pi \in \mathbb{Z}^2$ is normal to the
direction of the split. Similarly, for general MILPs, the
equations can be combined into a single equality from which a split
inequality is derived. Split inequalities are equivalent to Gomory
mixed integer cuts \cite{nw}. Empirical evidence shows that split
inequalities can be effective for strengthening the linear
programming relaxation of MILPs \cite{balsax,dgl}. Interestingly,
triangle and quadrilateral inequalities cannot be derived from a
single equation \cite{alww}. They can only be derived from (\ref{SI}) without
aggregating the two equations. Recent computational experiments by
Espinoza \cite{esp} indicate that quadrilaterals also induce
effective cutting planes in the context of solving general MILPs. In
this paper, we consider the relative strength of split, triangle and
quadrilateral inequalities from a theoretical point of view. We use
an approach for measuring strength initiated by Goemans
\cite{goemans}, based on the following definition and results.

Let $Q \subseteq \mathbb{R}_+^n \setminus \{ 0 \}$ be a polyhedron
of the form $Q = \{ x: \; a^i x \geq b_i \mbox{ for } i=1, \ldots ,
m \}$ where $a^i \geq 0$ and $b_i \geq 0$ for $i = 1 , \ldots , m$
and let $\alpha > 0$ be a scalar. We define the polyhedron $\alpha
Q$ as $\{ x: \; \alpha a^i x \geq b_i \mbox{ for } i=1, \ldots , m
\}$. Note that $\alpha Q$ contains $Q$ when $\alpha \geq 1$. It will
be convenient to define $\alpha Q$ to be $\mathbb{R}_+^n$ when
$\alpha = + \infty$.

We need the following generalization of a theorem of Goemans \cite{goemans}.

\begin{theorem} \label{THM:goemans}
Suppose $Q \subseteq \mathbb{R}_+^n \setminus \{ 0 \}$ is defined as
above. If convex set $P \subseteq \mathbb{R}^n_+$ is a relaxation of
$Q$ (i.e. $Q \subseteq P$), then the smallest value of $\alpha \ge 1$
such that $P \subseteq \alpha Q$ is
$$\max_{i=1, \ldots, m} \left\{\frac{b_i}{\inf\{ a^i x \ : \ \ x \in P\}} \ :
\ b_i > 0\ \right\}.$$ Here, we define $\frac{b_i}{\inf\{ a^i x \ :
\ \ x \in P\}}$ to be $+\infty$ if $\inf\{ a^i x \ : \ \ x \in P\} =
0$.
\end{theorem}

In other words, the only directions that need to be considered to
compute $\alpha$ are those defined by the nontrivial facets of $Q$.
Goemans' paper assumes that both $P$ and $Q$ are polyhedra, but one
can easily verify that only the polyhedrality of $Q$ is needed in
the proof. We give the proof of Theorem \ref{THM:goemans} in Section
\ref{SEC:goemans_proof}, for completeness.

\subsection{Results}\label{SEC:Results}

Let the {\it split closure} $S_f(r^1,\ldots,r^k)$ be the
intersection of all split inequalities, let the {\it triangle
closure} $T_f(r^1,\ldots,r^k)$ be the intersection of all
inequalities arising from maximal lattice-free triangles, and let
the {\it quadrilateral closure} $Q_f(r^1,\ldots,r^k)$ be the
intersection of all inequalities arising from maximal lattice-free
quadrilaterals. In the remainder of the paper, to simplify notation,
we refer to $R_f(r^1,\ldots,r^k)$, $S_f(r^1,\ldots,r^k)$,
$T_f(r^1,\ldots,r^k)$ and $Q_f(r^1,\ldots,r^k)$ as $R_f^k, S_f^k,
T_f^k$ and $Q_f^k$ respectively, whenever the rays $r^1,\ldots,r^k$
are obvious from the context.

Since all the facets of $R_f^k$ are induced by these three families
of maximal lattice-free convex sets, we have
$$R_f^k=S_f^k \cap T_f^k \cap Q_f^k.$$
It is known that the split closure is a polyhedron (Cook, Kannan and
Schrijver \cite{COOK}) but such a result is not known for the
triangle closure and the quadrilateral closure.
In this paper we show the following results.

\begin{theorem} \label{THM:inclusion}
$T_f^k \subseteq S_f^k$ and $Q_f^k \subseteq S_f^k.$
\end{theorem}

This theorem may seem counter-intuitive because some split
inequalities are facets of $R_f^k$. However, we show that any split
inequality can be obtained as the limit of an infinite collection of
triangle inequalities. What we will show is that, if a point is cut
off by a split inequality, then it is also cut off by some triangle
inequality. Consequently, the intersection of \emph{all} triangle
inequalities is contained in the split closure. The same is true of
the quadrilateral closure.

\begin{example}
\label{TfEx} As an illustration, consider the simple example in
Figure \ref{FIG:TfEx}. The split inequality is $s_1 \ge 1$ and the
split closure is given by $S_f^2 = \{(s_1, s_2) \ | \ s_1 \ge 1, s_2
\ge 0\}$. All triangle inequalities are of the form $a s_1 + b s_2
\ge 1$ with $a \ge 1$ and $b > 0$. The depicted triangle inequality
is of the form $s_1 + 2 s_2 \ge 1$ and by moving the corner $v$
closer to the boundary of the split, one can get a triangle
inequality of the form $s_1 + b s_2 \ge 1$ with $b$ tending to 0,
but remaining positive. Note that the split inequality can not be
obtained as a positive combination of triangle inequalities, but any
point cut by the split inequality is cut by one of the triangle
inequalities. As a result, the triangle closure is $T_f^2 = \{(s_1,
s_2) \ | \ s_1 + b s_2 \ge 1, \mbox{\ for \  all \ } b > 0, s_1 \ge
0, s_2 \ge 0\} = S_f^2$. \flushright\qed
\end{example}

\begin{figure}[htbp]
\centering \scalebox{0.45}{\epsfig{file=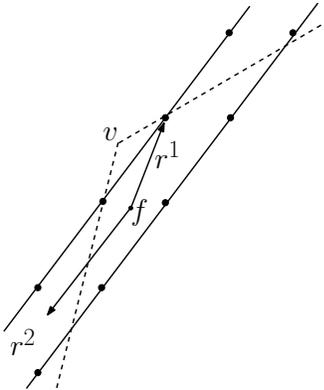}}
\caption{Illustration for Example \ref{TfEx}} \label{FIG:TfEx}
\end{figure}

We further study the strength of the triangle closure and
quadrilateral closure in the sense defined in Section
\ref{SUB:MOTIVATION}. We first compute the strength of a single Type
1 triangle facet as $f$ varies in the interior of the triangle,
relative to the entire split closure.

\begin{theorem}
\label{THM:T1split} Let $T$ be a Type 1 triangle as depicted in
Figure \ref{FIG:T1split}. Let $f$ be in its interior and assume that
the set of rays $\{r^1,\ldots ,r^k\}$ contains rays pointing to the
three corners of $T$. Let $\sum_{i=1}^{k} \psi(r^i) s_i\ge 1$ be the
inequality generated by $T$. The value
$$
\min\left\{ \sum_{i=1}^k \psi(r^i)s_i \ : \; s \in S_f^k \right\}
$$
is a piecewise linear function of $f$ for which some level curves
are depicted in Figure \ref{FIG:T1split}. This function varies from
a minimum of $\frac{1}{2}$ in the center of $T$ to a maximum of
$\frac{2}{3}$ at its corners.
\end{theorem}

\begin{figure}[htbp]
\centering \scalebox{0.7}{\epsfig{file=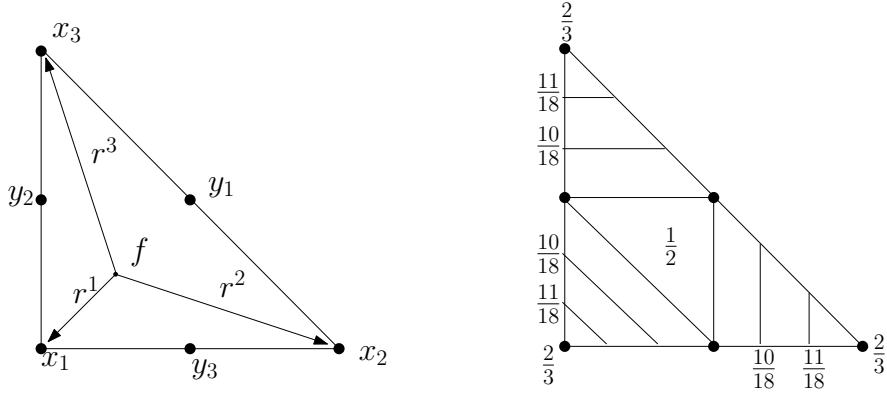}}
\caption{Illustration for Theorem
\ref{THM:T1split}}\label{FIG:T1split}
\end{figure}

Next we show that both the triangle closure and the quadrilateral
closure are good approximations of the integer hull $R_f^k$ in the
sense that

\begin{theorem} \label{THM:T&Q} $\;$

$R_f^k \subseteq T_f^k \subseteq 2 R_f^k$ and

$R_f^k \subseteq Q_f^k \subseteq 2 R_f^k.$
\end{theorem}

Finally we show that the split closure may not be a good
approximation of the integer hull.

\begin{theorem} \label{THM:T&Q_2} For any $\alpha > 1$, there is a
choice of $f$, $r^1, \ldots , r^k$ such that

$S_f^k \not\subseteq \alpha R_f^k.$
\end{theorem}


These results provide additional support for the recent interest in
cuts derived from two or more rows of an integer program
\cite{alww,bc,cm,dw,esp, go07}.

\section{Proof of Theorem
\ref{THM:goemans}}\label{SEC:goemans_proof}

\begin{proof}
Let
$$\alpha = \max_{i=1, \ldots, m} \left\{\frac{b_i}{\inf\{ a^i x \ : \ \ x \in P\}} \ :
\ b_i > 0\ \right\}.$$
We first show that $P \subseteq \alpha Q$. This holds when $\alpha =
+ \infty$ by definition of $\alpha Q$. Therefore we may assume $1 \le
\alpha < + \infty$.  Consider any point $p\in P$. The inequalities
of $\alpha Q$ are of the form $\alpha a^i x \geq b_i$ with $a^i \geq 0$
and $b_i \geq 0$. If $b_i = 0$, then since $p \in P \subseteq
\mathbb{R}^n_+$, $a^i p \geq 0$ and hence this inequality is
satisfied. If $b_i > 0$, then we know from the definition of
$\alpha$ that

\[
\frac{b_i}{\inf\{ a^i x \ : \ \ x \in P\}} \leq \alpha .
\]
This implies
\[
b_i \leq \alpha \inf\{ a^i x \ : \ \ x \in P\} \leq \alpha a^i p .
\]
Therefore, $p$ satisfies this inequality.

We next show that for any $1 \le \alpha' < \alpha$, $P \not\subseteq
\alpha' Q$. Say $\alpha = \frac{b_j}{\inf\{ a^j x \ : \ \ x \in
P\}}$ (i.e. the maximum, possibly $+ \infty$, is reached for index
$j$). Let $\delta = \frac{b_j}{\alpha'} - \frac{b_j}{\alpha}$. We
have $\delta > 0$. From the definition of $\alpha$ we know that
$\inf\{ a^j x \ : \ \ x \in P\} = \frac{b_j}{\alpha}$. Therefore,
there exists $p \in P$ such that $a^j p < \frac{b_j}{\alpha} +
\delta = \frac{b_j}{\alpha'}$. So $\alpha' a^j p < b_j$ and hence $p
\not\in \alpha' Q$.

\end{proof}

\section{Split closure vs. triangle and quadrilateral
closures}

In this section, we present the proof of
Theorem~\ref{THM:inclusion}.

\begin{proof} (Theorem \ref{THM:inclusion}).
We show that if any point $\bar{s}$ is cut off by a split
inequality, then it is also cut off by some triangle inequality.
This will prove the theorem.

Consider any split inequality $\sum_{i=1}^k \psi_S(r^i)s_{i} \ge 1$
(see Figure~\ref{fig:split_with triangle}) and denote by $L_1$ and
$L_2$ its two boundary lines. Point $f$ lies in some parallelogram
of area 1 whose vertices $y^1, y^2, y^3, \textrm{and }y^4$ are
lattice points on the boundary of the split.

Assume without loss of generality that $y^1$ and $y^2$ are on $L_1$.
Consider the family $\mathcal{T}$ of triangles whose edges are
supported by $L_2$ and by two lines passing through $y^1$ and $y^2$
and whose interior contains the segment $y^1y^2$. See
Figure~\ref{fig:split_with triangle}. Note that all triangles in
$\mathcal{T}$ are of Type $2$. For $T \in \mathcal{T}$ we will
denote by $\psi_T$ the minimal function associated with $T$.

\begin{figure}[htbp]
   \begin{center}
     \scalebox{.5}{\epsfig{file=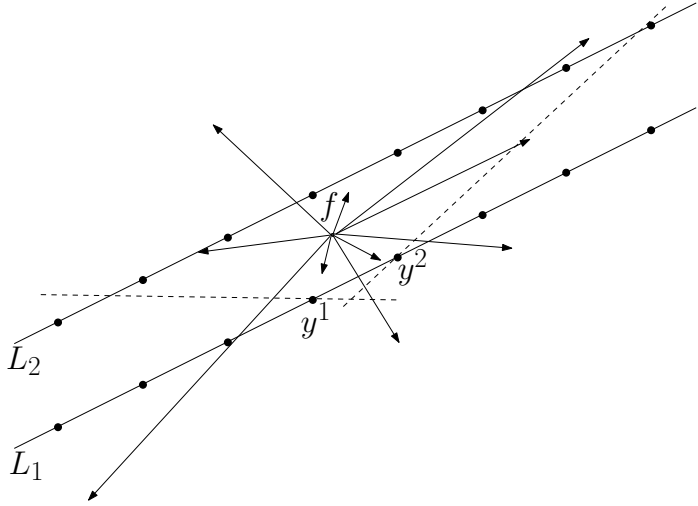}}
\caption{Approximating a split inequality with a triangle
inequality. The triangle is formed by $L_2$ and the two dashed lines}
\label{fig:split_with triangle}
   \end{center}
\end{figure}

By assumption, $\sum_{i=1}^{k} \psi_S(r^i)\bar{s}_i < 1$. Let
$\epsilon = 1 - \sum_{i=1}^{k} \psi_S(r^i)\bar{s}_i$.

We now make the following simple observation. Given a finite set $X$
of points that lie in the interior of the split $S$, we can find a
triangle $T \in \mathcal{T}$ as defined above, such that all points
in $X$ are in the interior of $T$. To see this, consider the convex
hull $\mathcal{C}(X)$ of $X$. Since all points in $X$ are in the
interior of $S$, so is $\mathcal{C}(X)$. This implies that the
tangent lines from $y^1$ and $y^2$ to $\mathcal{C}(X)$ are not
parallel to $L_1$. Two of these four tangent lines along with $L_2$
of $S$ form a triangle in $\mathcal{T}$ with $X$ in its interior.

Let $s_{max} = \max \{\bar{s}_i \ : \ i = 1 \ldots, k\}$ and define
$\delta = \frac{\epsilon}{2 \cdot k \cdot s_{max}} > 0$. For every
ray $r^i$ define $c(r^i) = \psi_S(r^i) + \delta$. Therefore, by
definition $p^i = f + \frac{1}{c(r^i)} \cdot  r^i$ is a point
strictly in the interior of $S$. Using the observation stated above,
there exists a triangle $T\in \mathcal{T}$ which contains all the
points $p^i$. It follows that the coefficient $\psi_T(r^i)$ for any
ray $r^i$ is less than or equal to $c(r^i)$.

We claim that for this triangle $T$ we have $\sum_{i=1}^{k}
\psi_T(r^i)\bar{s}_i < 1$. Indeed,

\begin{eqnarray*}
\sum_{i=1}^{k} \psi_T(r^i)\bar{s}_i & \leq & \sum_{i=1}^k c(r^i)\bar{s}_i \\
 & = & \sum_{i=1}^k(\psi_S(r^i) + \delta) \bar{s}_i
=  \sum_{i=1}^k\psi_S(r^i)\bar{s}_i + \sum_{i=1}^k\frac{\epsilon}{2 \cdot k \cdot s_{max}}\bar{s}_i\\
 & \leq & \sum_{i=1}^k\psi_S(r^i)\bar{s}_i + \frac{\epsilon}{2} =
1 - \frac{\epsilon}{2} <  1
\end{eqnarray*}

The first inequality follows from the definition of $c(r^i)$ and the
last equality follows from the fact that
$\sum_{i=1}^k\psi_S(r^i)\bar{s}_i = 1 - \epsilon$.

This shows that $T_f^k \subseteq S_f^k$. For the quadrilateral
closure, we also use two lines passing through $y^3$ and $y^4$ on
$L_2$ and argue similarly.
\end{proof}

Note that even though there can be a zero coefficient in a split
inequality for some ray $r$, in the proof above we exhibit a
sequence of triangle inequalities with arbitrarily small
coefficients for ray $r$. Any point cut off by the split inequality
is also cut off by a cut in the sequence.

\section{Tools}
\label{SEC:PRELIM}

\subsection{Conditions under which a maximal lattice-free convex set gives rise to a facet}

Andersen, Louveaux, Weismantel and Wolsey \cite{alww}
 characterized the facets of $R_f^k$ as arising from
 splits, triangles and quadrilaterals.
 Cornu\'ejols and Margot \cite{cm} gave a converse.
 We give this characterization in Theorem \ref{THM:mainRfk}. Roughly speaking, for a maximal lattice-free triangle or quadrilateral to give rise to a facet,
 it  has to have  its corner points on half-lines $f+\lambda r^j$ for some $j =
1, \ldots , k$ and $\lambda > 0$; or to
satisfy a certain technical condition called the ray condition.
 Although the ray condition is not central to this paper (it is only used once
 in the proof of Theorem \ref{thm:7.2}), we need to include it for technical completeness.

 Let $B_\psi$ be a maximal lattice-free split,
triangle or quadrilateral with
$f$ in its interior. For any $j=1, \ldots, k$ such that $\psi(r^j) > 0$,
let $p^j$ be the intersection  of the half-line
$f + \lambda r^j$, $\lambda \geq 0$, with the boundary of $B_\psi$.
The point $p^j$ is called the {\it boundary point} for $r^j$. Let
$P$ be a set of boundary points. We say that a point $p \in P$ is
{\it active} if it can have a positive coefficient in a convex
combination of points in $P$ generating an integral point. Note that
$p \in P$ is active if and only if $p$ is integral or there exists
$q \in P$ such that the segment $pq$ contains an integral point in
its interior. We say that an active point $p \in P$ is {\it uniquely
active} if it has a positive coefficient in {\it exactly one} convex
combination of points in $P$ generating an integral point.

Apply the following {\it Reduction Algorithm}:
\begin{itemize}
\item[0.)]
Let $P = \{ p^1, \ldots , p^k \}$.

\item[1.)]
While there exists $p \in P$ such that $p$ is active and $p$
is a convex combination of
other points in $P$, remove $p$ from $P$.
At the end of this step, $P$ contains at
most two active points on each edge of $B_\psi$ and all points of $P$ are distinct.

\item[2.)]
While there exists a uniquely active $p \in P$, remove $p$ from $P$.

\item[3.)]
If $P$ contains exactly two active points $p$ and $q$
(and possibly inactive points), remove both $p$ and $q$ from $P$.
\end{itemize}

\medskip
{\em The ray condition holds for a triangle or a quadrilateral}
if $P = \emptyset$ at termination of the Reduction Algorithm.

{\em The ray condition holds for a
split} if, at termination of the Reduction Algorithm, either $P =
\emptyset$, or $P= \{p_1, q_1, p_2, q_2\}$ with $p_1, q_1$ on one of
the boundary lines and $p_2, q_2$ on the other and both line
segments $p_1q_1$ and $p_2q_2$ contain at least two integral points.

\begin{theorem} \label{THM:mainRfk}
{\em (Cornu\'ejols and Margot \cite{cm})} The facets of $R_f^k$ are
\begin{itemize}
\item[(i)] split inequalities where the unbounded direction of $B_\psi$ is
$r^j$ for some $j = 1, \ldots , k$ and the line $f+\lambda r^j$
contains no integral point; or where $B_\psi$ satisfies the
ray condition,
\item[(ii)] triangle inequalities where the triangle $B_\psi$ has
its corner points on three half-lines $f+\lambda r^j$ for some $j =
1, \ldots , k$ and $\lambda > 0$; or where the triangle $B_\psi$
satisfies the ray condition,
\item[(iii)] quadrilateral inequalities where  the corners of $B_\psi$ are
on four half-lines $f+\lambda r^j$ for some $j = 1, \ldots , k$ and
$\lambda > 0$, and $B_\psi$ satisfies a certain ratio condition (the ratio
condition will not be needed in this paper; the interested reader is referred
to \cite{cm} for details).
\end{itemize}
\end{theorem}

Note that the same facet may arise from different convex sets. For
example quadrilaterals for which the ray condition holds define
facets, but there is always also a triangle that defines the same
facet, which is the reason why there is no mention of the ray
condition in (iii) of Theorem  \ref{THM:mainRfk}.

\subsection{Reducing the number of rays in the analysis}

The following technical theorem will be used in the proofs of
Theorems~\ref{THM:T1split}, \ref{THM:T&Q} and \ref{THM:T&Q_2}, where
we will be applying Theorem~\ref{THM:goemans}.

\begin{theorem}
\label{THM:corner} Let $B_1, \ldots, B_m$ be lattice-free convex
sets with $f$ in the interior of $B_p$, $p=1,\ldots,m$. Let $R_c
\subseteq \{1,\ldots,k\}$ be a subset of the ray indices such that
for every ray $r^j$ with $j\not\in R_c$, $r^j$ is the convex
combination of some two rays in $R_c$. Define

\[
\begin{array}{rlcl}
  z_1 =    \;\; \min &
\displaystyle
\sum_{i=1}^k s_i \\[0.1in]
& \displaystyle \sum_{i=1}^k \psi_{B_p}(r^i) s_i & \geq 1 &
\mbox{ for \ } p = 1, \ldots, m  \\[0.1in]
          &  s \geq 0
\end{array}
\]

and

\[
\begin{array}{rlcl}
  z_c = \;\; \min &
\displaystyle \sum_{i \in R_c} s_i \\[0.1in]
& \displaystyle \sum_{i \in R_c} \psi_{B_p}(r^i) s_i & \geq 1 &
\mbox{ for \ } p = 1,\ldots,m \\[0.1in]
          &  s \geq 0 \ .
\end{array}
\]

Then $z_1 = z_c$.
\end{theorem}

\begin{proof}
Assume there exists $j \not \in R_c$ and $r^1, r^2$ are the rays in
 $R_c$ such that $r^j = \lambda r^1 + (1-\lambda)r^2$ for some $0 <
\lambda <1$. Let $K = \{1, \ldots, k\} - j$ and define

\[
\begin{array}{rlcl}
  z_2 =    \;\; \min & \displaystyle \sum_{i \in K} s_i \\[0.1in]
         &       \displaystyle \sum_{i\in K} \psi_{B_p}(r^i) s_i & \geq 1 &
\mbox{ for \ }p = 1, \ldots, m  \\[0.1in]
          &  s \geq 0 \ .
\end{array}
\]
We first show that $z_1 = z_2$. Applying the same reasoning
repeatedly to all the indices not in $R_c$ yields the proof of the
theorem.

Any optimal solution for the LP defining $z_2$ yields a feasible
solution for the one defining $z_1$ by setting $s_j = 0$, implying
$z_1 \leq z_2$. It remains to show that $z_1 \geq z_2$.

Consider any point $\hat{s}$ satisfying $\sum_{i=1}^k
\psi_{B_p}(r^i)\hat{s}_i \geq 1$ for every $p \in \{1,\ldots,m\}$.
Consider the following values $\bar{s}$ for the variables
corresponding to the indices $t \in K$.

\[
\bar{s}_t =  \left\{\begin{array}{ll} \hat{s}_t & \mbox{if \ }
t \not \in \{1, 2, j\} \\
\hat{s}_{1} + \lambda\hat{s}_{j} & \mbox{if \ } t = 1\\
\hat{s}_{2} + (1-\lambda)\hat{s}_{j} & \mbox{if \ } t = 2
\end{array}\right.
\]

One can check that
\[
\sum_{i \in K} \bar{s}_i = \hat{s}_j + \sum_{i \in K} \hat{s}_{i} \
.
\]

By Theorem~\ref{THM:minfunc} $\psi_{B_p}$ is convex, thus
$\psi_{B_p}(r^j) \leq \lambda\psi_{B_p}(r^1) +
(1-\lambda)\psi_{B_p}(r^2)$ for $p = 1, \ldots, m$. It follows that
$\sum_{i\in K}\psi_{B_p}(r^i)\bar{s}_{i} \geq
\psi_{B_p}(r^j)\hat{s}_j + \sum_{i \in K}\psi_{B_p}(r^i)\hat{s}_{i}
= \sum_{i=1}^k \psi_{B_p}(r^i)\hat{s}_i \geq 1$ for $p = 1, \ldots,
m$. Hence $\bar{s}$ satisfies all the inequalities restricted to
indices in $K$ and has the same objective value as $\hat{s}$. It
follows that $z_1 \geq z_2$.
\end{proof}

\section{Proof sketch for Theorems~\ref{THM:T1split}
and~\ref{THM:T&Q}}\label{SEC:proof_sketch}

In this section, we give a brief outline of the proofs of
Theorems~\ref{THM:T1split} and~\ref{THM:T&Q}.
A complete proof will be given in Sections \ref{SEC:T1split} and \ref{Section7} respectively.

In Theorem~\ref{THM:T1split}, we need to analyze the optimization problem

\begin{equation}\label{eq:Type1_over_Sf}
\min\left\{ \sum_{i=1}^k \psi(r^i) s_i \ : \; s \in S_f^k\right\} \
,
\end{equation}
\noindent
where $\psi$ is the minimal function derived from the Type 1 triangle.

For Theorem~\ref{THM:T&Q}, recall that all nontrivial facet defining
inequalities for $R_f^k$ are of the form $a^i s \geq 1$ with $a^i
\geq 0$. Therefore, Theorem \ref{THM:goemans} shows that to prove
Theorem~\ref{THM:T&Q}, we need to consider all nontrivial facet
defining inequalities and optimize in the direction $a^i$ over the
triangle closure $T_f^k$ and the quadrilateral closure $Q_f^k$. This
task is made easier since all the nontrivial facets of $R_f^k$ are
characterized in Theorem~\ref{THM:mainRfk}. Moreover,
Theorem~\ref{THM:inclusion} shows that every split inequality
$\sum_{i=1}^k \psi(r^i)s_i \geq 1$ is valid for $T_f^k$. Therefore,
if we optimize over $T_f^k$ in the direction $\sum_{i=1}^k
\psi(r^i)s_i$, we get a value of at least one.
Theorem~\ref{THM:inclusion} shows that this also holds for $Q_f^k$.
Thus, we can ignore the facets defined by split inequalities.

Formally, consider a maximal lattice-free triangle or quadrilateral
$B$ with associated minimal function $\psi$ that gives rise to a
facet $\sum_{j=1}^{k} \psi(r^j) s_j \ge 1$ of $R_f^k$. We want to
investigate the following optimization problems:

\begin{equation}\label{eq:facet_over_Tf}
\displaystyle \inf \left\{ \sum_{j=1}^{k} \psi(r^j) s_j : \; s \in
T_f^k \right\}
\end{equation}

and

\begin{equation}\label{eq:facet_over_Qf}
\displaystyle \inf \left\{ \sum_{j=1}^{k} \psi(r^j) s_j : \; s \in
Q_f^k \right\} .
\end{equation}

We first observe that, without loss of generality, we can make the
following simplifying assumptions for problems
(\ref{eq:Type1_over_Sf}), (\ref{eq:facet_over_Tf}) and
(\ref{eq:facet_over_Qf}).

\begin{assumption}\label{ass:scaling}
Consider the objective function $\psi$ in problems
(\ref{eq:Type1_over_Sf}), (\ref{eq:facet_over_Tf}) and
(\ref{eq:facet_over_Qf}). For every $j$ such that $\psi(r^j)
> 0$, the ray $r^j$ is such that the point $f+r^j$ is on the
boundary of the lattice-free set $B$ generating $\psi$.
\end{assumption}

Indeed, this amounts to scaling the coefficient for the ray $r^j$ by
a constant factor in every inequality derived from all maximal
lattice-free sets, including $B$. Therefore, this corresponds to a
simultaneous scaling of variable $s_j$ and corresponding
coefficients in problems (\ref{eq:Type1_over_Sf}),
(\ref{eq:facet_over_Tf}) and (\ref{eq:facet_over_Qf}). This does not
change the optimal values of these problems. Moreover, Cornu\'ejols
and Margot \cite{cm} show that the equations of all edges of
triangles of Type 1, 2, or 3, of quadrilaterals and the direction of
all splits generating facets of $R_f^k$ are rational. This implies
that the scaling factor for ray $r^j$ is a rational number and that
the scaled ray is rational too.

As a consequence, we can assume that the objective function of
problems (\ref{eq:Type1_over_Sf}), (\ref{eq:facet_over_Tf}) or
(\ref{eq:facet_over_Qf}) is  $\sum_{j=1}^k s_j$.

When $B_{\psi}$ is a triangle or quadrilateral and $f$ is in its
interior, define a {\it corner ray} to be a ray $r$ such that $f+
\lambda r$ is a corner of $B_{\psi}$ for some $\lambda > 0$.

\begin{remark}\label{rem:simplify}
If $\{r^1, \ldots, r^k\}$ contains the corner rays of the convex set
defining the objective functions of
(\ref{eq:Type1_over_Sf}), (\ref{eq:facet_over_Tf}) or
(\ref{eq:facet_over_Qf}), then Assumption~\ref{ass:scaling} implies
that the hypotheses of Theorem~\ref{THM:corner} are satisfied.
Therefore, when analyzing
(\ref{eq:Type1_over_Sf}), (\ref{eq:facet_over_Tf}) or
(\ref{eq:facet_over_Qf}), we can assume that $\{r^1, \ldots, r^k\}$ is exactly
the set of corner rays.
\end{remark}

\section{Type 1 triangle and the split closure}
\label{SEC:T1split}

In this section, we present the proof of Theorem \ref{THM:T1split}.

Consider any Type $1$ triangle $T$ with integral vertices $x^j$, for
$j=1, 2, 3$, and one integral point $y^j$ for $j = 1, 2, 3$ in the
interior of each edge. We want to study the optimization problem
(\ref{eq:Type1_over_Sf}). Recall that Remark~\ref{rem:simplify} says
that we only need to consider the case with three corner rays $r^1$,
$r^2$ and $r^3$.

We compute the exact value of

\begin{equation} \label{Split3}
\begin{array}{rlcl}
  z_{SPLIT} =    \;\; \min & \displaystyle \sum_{j=1}^3 s_j \\[0.1in]
& \displaystyle \sum_{j=1}^{3} \psi(r^j) s_j & \geq 1 &
\mbox{ for all splits } B_\psi \\[0.1in]
&  s  \in  \mathbb{R}_+^3 .
\end{array}
\end{equation}

Observe that, using a unimodular transformation, $T$ can be made to
have one horizontal edge $x^1x^2$ and one vertical edge $x^1x^3$, as
shown in Figure \ref{FIG:T1split}. Without loss of generality, we
place the origin at point $x^1$.

We distinguish two cases depending on the position of $f$ in the
interior of triangle $T$: $f$ is in the inner triangle $T_I$ with
vertices $y^1 = (1,1)$, $y^2 = (0,1)$ and $y^3 = (1,0)$; and $f \in
int(T) \setminus T_I$. We show that $z_{SPLIT} = \frac{1}{2}$ when
$f$ is in the inner triangle $T_I$ and that $z_{SPLIT}$ increases
linearly from $\frac{1}{2}$ when $f$ is at the boundary of $T_I$ to
$\frac{2}{3}$ at the corners of the triangle $T$ when $f \in int(T)
\setminus T_I$. See the right part of Figure \ref{FIG:T1split} for some
level curves of $z_{SPLIT}$ as a function of the position of $f$ in
$T$. By a symmetry argument, it is sufficient to consider the inner
triangle $T_I$  and the corner triangle $T_C$ defined by $f_1+f_2
\leq 1$, $f_1,f_2 \geq 0$.

\begin{theorem} \label{THM:TrIntVer}
Let $T$ be a Type 1 triangle with integral vertices, say $(0,0)$,
$(0,2)$ and $(2,0)$. Then
\begin{itemize}
\item[(i)]
$z_{SPLIT} = \frac{1}{2}$ when $f$ is interior to the triangle with
vertices $(1,0)$, $(0,1)$ and $(1,1)$.
\item[(ii)]
$z_{SPLIT} = 1-\frac{1}{3-f_1-f_2}$ when $f=(f_1,f_2)$ is interior
to the corner triangle $f_1+f_2 \leq 1$, $f_1,f_2 \geq 0$. The value
of $z_{SPLIT}$ when $f$ is in the other corner triangles follows by
symmetry.
\end{itemize}
\end{theorem}

To prove this theorem, we show that the split closure is completely
defined by only three split inequalities. In other words, all other
split inequalities are dominated by these three split inequalities.

Define $S_1$ as the convex set $1 \le x_1+x_2 \leq 2$,
$S_2$ as the convex set $0 \le x_1 \leq 1$ and $S_3$
as the convex set $0 \le x_2 \leq 1$. Define {\it Split 1} (resp.
{\it Split 2}, {\it Split3}) to  be the inequality obtained from $S_1$
(resp. $S_2$, $S_3$).

Let $S$ be a split $c \leq a x_1 + b a_2 \leq c+1$ with $f$ in the interior of $S$.  The {\it shores} of $S$
are the two half-planes $a x_1 + b a_2 \leq c$ and $a x_1 + b a_2 \geq c+1$.

\begin{remark} \label{LEM:OppSides}
Let $A$, $B$, and $C$ be three points on a line, with $B$ between
$A$ and $C$ and let $S$ be a split.
If $A$ and $C$ are not in the interior of $S$ but $B$ is, then $A$ and $C$ are
on opposite shores of $S$.
\end{remark}


\begin{figure}[htbp]
\centering \scalebox{0.5}{\epsfig{file=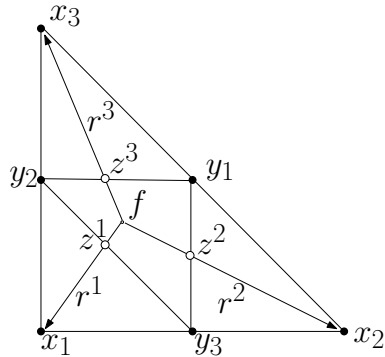}}
\caption{Illustration for the proof of Lemma \ref{Splits123}}
\label{FIG:Splits123}
\end{figure}

\begin{lemma} \label{Splits123}
If $f$ is in the interior of the triangle $T_I$ with vertices
$(0,1)$, $(1,0)$ and $(1,1)$, then the split closure is defined by
Split 1, Split 2 and Split 3.
\end{lemma}

\begin{proof}
Let $S$ be a split defining an inequality that is not dominated by either of
Split 1, Split 2, or Split 3.
For $i = 1, 2, 3$, let $z^i$ (resp. $w^i$)
be the boundary point for $r^i$ on the boundary of $T_I$ (resp. $S$)
(see Figure \ref{FIG:Splits123}). (Note that since $x^i = f + r^i$
is integer, these points exist.) Observe that if
$z^1$ is not in the interior of $S$, then the inequality obtained from $S$ is
dominated by Split 1, since the three boundary points $w_1, w_2, w_3$
are closer to $f$ than the corresponding three boundary points $z^1,
x^2, x^3$ for the three rays on the boundary of $S_1$.
A similar observation holds for $z^2$ and Split 2
and for $z^3$ and Split 3, yielding that $z^2$ and $z^3$ are also in
the interior of $S$.

Since the points $y^1, y^2$ and $y^3$ are integer, they are not
in the interior of $S$.
Applying Remark \ref{LEM:OppSides} to $y^1, z^3, y^2$, we must have that $y^1$ and $y^2$ are on opposite shores of
$S$. Now, $y^3$ is in one of the two shores of $S$. Assume without
loss of generality that it is on the same shore as $y^1$.
Applying Remark \ref{LEM:OppSides} to $y^1, z^2, y^3$, we have that $y^1$ and $y^3$ are on opposite shores of
$S$, a contradiction.
\end{proof}

\begin{lemma}
If $f$ is in the interior of the triangle $T_I$ with vertices
$(0,1)$, $(1,0)$ and $(1,1)$, then $z_{SPLIT} = \frac{1}{2}$.
\end{lemma}

\begin{proof} By Lemma \ref{Splits123}, $z_{SPLIT}$ is given by

\begin{equation}
\begin{array}{rlcl}
  z_{SPLIT} =   \;\; \min & \displaystyle \sum_{j=1}^3 s_j \\[0.1in]
&\displaystyle \sum_{j=1}^{3} \psi_i(r^j) s_j & \geq 1 &
\mbox{ for } i=1,2,3 \\[0.1in]
&  s  \in  \mathbb{R}_+^3
\end{array}
\end{equation}

\noindent where $\psi_i(r^j)$ is the coefficient of $s_j$ in Split
$i$, for $i = 1, 2, 3$. Let $f=(f_1, f_2)$. The coefficient of $s^j$
in the split inequality can be computed from the boundary point for
$r^j$ with the corresponding split. For example, the boundary points
for $r^2$ and $r^3$ with $S_1$ are the integer points $x^2$ and
$x^3$. This implies that $\psi_1(r^2)=\psi_1(r^3)=1$. On the other
hand, the boundary point for $r^1$ is the point $t=
\left(\frac{f_1}{f_1+f_2}, \frac{f_2}{f_1+f_2} \right)$. The length
of $r^1$ divided by the length of the segment $ft$ determines the
coefficient $\psi_1(r^1)$ of $s_1$ (This follows from the
homogeneity of $\psi_1$ and the fact that $\psi_1(t - f)=1$ since
$t$ is on the boundary of $S_1$). We get
$\psi_1(r^1)=\frac{f_1+f_2}{f_1+f_2-1}$. Repeating this for $S_2$
and $S_3$, we get that $z_{SPLIT}$ is the optimal value of the
following linear program.

\begin{equation} \label{Split123}
\begin{array}{rrrrl}
  z_{SPLIT} =    \;\; \min & s_1 & + s_2 & + s_3 \\
  &  \frac{f_1+f_2}{f_1+f_2-1} s_1  & + s_2 & + s_3 &  \geq 1  \\
   &  s_1  & + \frac{2-f_1}{1-f_1} s_2 & + s_3 & \geq 1  \\
    &  s_1  & + s_2 & + \frac{2-f_2}{1-f_2} s_3 & \geq 1  \\
  &  &  s  \geq 0.
\end{array}
\end{equation}

Its optimal solution $s^*$ is $s^*_1 = \frac{f_1+f_2-1}{2}$, $s^*_2 =
\frac{1-f_1}{2}$, $s^*_3 = \frac{1-f_2}{2}$ with value $z_{SPLIT} =
s^*_1 + s^*_2 + s^*_3 = \frac{1}{2}$. Indeed,
note that all three inequalities in (\ref{Split123}) are satisfied
at equality and that the dual of (\ref{Split123}) is feasible (for example,
$(0,0,0)$ is a solution).
Therefore the complementary slackness conditions hold for $s^*$ with any
feasible solution of the dual.


\end{proof}

Now we prove the second part of the theorem, when $f$ is interior to
the corner triangle with vertices $(0,0)$, $(1,0)$ and $(0,1)$ or an
inner point on the segment $y^2y^3$.

\begin{lemma} \label{Splits23}
If $f$ is in the interior of the triangle with vertices $(0,0)$,
$(0,1)$ and $(1,0)$, or an inner point on the segment joining
$(0,1)$ to $(1,0)$, then the split closure is defined by Split 2 and
Split 3.
\end{lemma}

\begin{proof}
Let $S$ be a split defining a split inequality that is not dominated
by either of Split 2, or Split 3. Let $z^2$ be the intersection point of $r^2$
with $y^1y^3$ and let $z^3$ be the intersection point of $r^3$ with
$y^1y^2$. For $i = 1, 2, 3$, let $w^i$ be the intersection point of $r^i$
with either $L_1$ or $L_2$. (Note that since $x^i$ is integer, $r^i$
has to intersect one of the two lines.) Observe that if $z^2$ is not
in the interior of $S$, then the inequality obtained from $S$ is
dominated by Split 2, since the three intersection points $w_1, w_2, w_3$
are closer to $f$ than the corresponding three intersection points $x^1,
z^2, x^3$ for $S_2$. A similar observation holds for $z^3$ and
$S_3$, yielding that $z^3$ is also in the interior of $S$.

Since the points $y^1, y^2$ and $y^3$ are integer, they are not
in the interior of $S$. Applying Remark \ref{LEM:OppSides} to $y^1, z^3, y^2$, we have that $y^1$ and $y^2$ are on opposite shores of $S$.
Applying Remark \ref{LEM:OppSides} to $y^1, z^2, y^3$, we have that $y^1$ and $y^3$ are on opposite shores of $S$. It
follows that $y^2$ and $y^3$ are on the same shore $W$ of $S$ and
thus the whole segment $y^2y^3$ is in $W$. This is a contradiction
with the fact that both $f$ and $z^3$ are in the interior of $S$, as
the two segments $y^2y^3$ and $fz^3$ intersect.
\end{proof}

\begin{lemma}
If $f = (f_1,f_2)$ is in the interior of the triangle with vertices
$(0,0)$, $(0,1)$ and $(1,0)$, or an inner point on the segment
joining $(0,1)$ to $(1,0)$, then $z_{SPLIT} = 1-
\frac{1}{3-f_1-f_2}$.
\end{lemma}

\begin{proof}
The optimal solution of the LP

\begin{equation} \label{Split12}
\begin{array}{rrrrl}
  z_{SPLIT} =       & \min s_1 & + s_2 & + s_3 \\
   &  s_1  & + \frac{2-f_1}{1-f_1} s_2 & + s_3 & \geq 1  \\
    &  s_1  & + s_2 & + \frac{2-f_2}{1-f_2} s_3 & \geq 1  \\
  &  &  s  \geq 0.
\end{array}
\end{equation}

\noindent is $s_1=0$, $s_2 = \frac{1-f_1}{3-f_1-f_2}$, $s_3 =
\frac{1-f_2}{3-f_1-f_2}$.
\end{proof}

This completes the proof of Theorem \ref{THM:TrIntVer}.
This theorem in conjunction with Theorem~\ref{THM:goemans} implies
that including all Type 1 triangle facets can
improve upon the split closure only by a factor of 2.

\begin{corollary}
Let $\mathcal{F}$ be the family of all facet defining inequalities
arising from Type 1 triangles. Define
$$
\bar S_f = S_f^k\cap \left\{\sum_{i=1}^{k} \psi(r^i)s_i \geq 1 \ : \
\psi \in \mathcal{F} \right\}\ .
$$
Then $\bar S_f \subseteq S_f^k \subseteq 2 \bar S_f$.
\end{corollary}

\section{Integer hull vs. triangle and quadrilateral closures}
\label{Section7}

In this section we present the proof of Theorem~\ref{THM:T&Q}. We
show that the triangle closure $T_f^k$ and the quadrilateral closure
$Q_f^k$ both approximate the integer hull $R_f^k$ to within a factor
of two. As outlined in Section~\ref{SEC:proof_sketch}, we can show
this by taking a facet of $R_f^k$, and optimizing in that direction
over $T_f^k$ or $Q_f^k$. As noted in that section, we need to
analyze the optimization problems (\ref{eq:facet_over_Tf}) and
(\ref{eq:facet_over_Qf}), where the objective function comes from a
maximal lattice-free triangle or quadrilateral.

\subsection{Approximating the integer hull by the triangle
closure}
\label{sub:IHbyT}

We only need to consider facets of $R_f^k$ derived from
quadrilaterals to obtain the objective function of problem
(\ref{eq:facet_over_Tf}). We prove the following result.

\begin{theorem}
Let $Q$ be a maximal lattice-free quadrilateral with corresponding
minimal function $\psi$ and generating a facet $\displaystyle
\sum_{i=1}^k \psi(r^i) s_i \geq 1$ of $R_f^k$. Then
$$
\inf\left\{ \displaystyle \sum_{i=1}^k \psi(r^i) s_i \ : \; s \in
T_f^k\right\} \geq \frac{1}{2} \ .
$$
\end{theorem}

\begin{proof}
The theorem holds if the facet defining inequality can also be
obtained as a triangle inequality. Therefore by
Theorem~\ref{THM:mainRfk}, we may assume that rays $r^1, \ldots r^4$
are the corners rays of $Q$ (see Figure~\ref{fig:quad_with
triangle}). We remind the reader of Remark~\ref{rem:simplify},
showing that we can assume that $k = 4$ and that the four rays are
exactly the corner rays of $Q$.

By a unimodular transformation, we may further assume that the four
integral points on the boundary of $Q$ are $(0,0), (1,0), (1,1),
(0,1)$. Moreover, by symmetry, we may assume that the fractional
point $f$ satisfies $f_1\leq \frac{1}{2}$ and $f_2\leq \frac{1}{2}$
as rotating this region  about $(\frac{1}{2},\frac{1}{2})$ by
multiples of $\frac{\pi}{2}$ covers the entire quadrilateral. Note
that $f_1 < 0$ and $f_2 < 0$ are possible.

We relax Problem  (\ref{eq:facet_over_Tf}) by keeping only two of
the triangle inequalities, defined by triangles $T_1$ and $T_2$.
$T_1$ has corner $f+r^4$ and edges supported by the two
edges of $Q$ incident with that corner and by the line $x=1$. $T_2$
has corner $f+r^1$ and edges supported by the two edges
of $Q$ incident with that corner and by the line $y=1$. The two
triangles are depicted in dashed lines in Figure~\ref{fig:quad_with
triangle}.

\begin{figure}[htbp]
   \begin{center}
     \scalebox{.5}{\epsfig{file=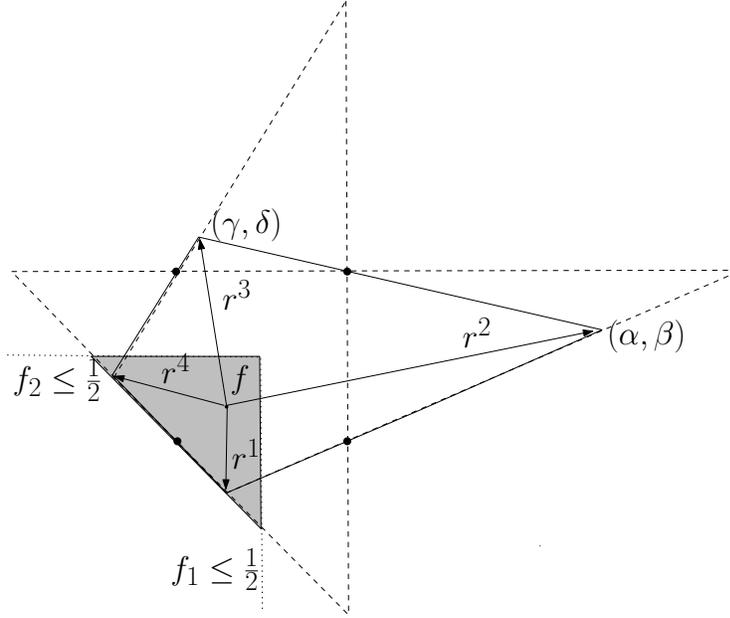}}
\caption{Approximating a quadrilateral inequality with triangle inequalities}
\label{fig:quad_with triangle}
   \end{center}
\end{figure}

Thus, Problem (\ref{eq:facet_over_Tf}) can be relaxed to the LP



\begin{equation}\label{LP7}
\begin{array}{rlcl}
  \;\; \min & s_1 + s_2 + s_3 + s_4 \\
         &  \displaystyle \sum_{i=1}^4 \psi_{T_1}(r^i)s_i & \geq 1 &
(\textrm{Triangle }T_1) \\[0.1in]
          &  \displaystyle \sum_{i=1}^4 \psi_{T_2}(r^i)s_i & \geq 1 &
(\textrm{Triangle }T_2) \\[0.1in]
&  s  \in  \mathbb{R}_+^4 .
\end{array}
\end{equation}

Let $(\alpha, \beta) = f + r^2$ and $(\gamma, \delta) = f + r^3$.
Computing the coefficients $\psi_{T_1}(r^2)$ and $\psi_{T_2}(r^3)$,
LP (\ref{LP7}) becomes

\begin{equation}\label{LP8}
\begin{array}{rlcl}
  \;\; \min & s_1 + s_2 + s_3 + s_4 \\
& s_1 + \displaystyle \frac{\alpha-f_1}{1-f_1}s_2 + s_3 + s_4 & \geq 1 &
(T_1) \\[0.1in]
&  s_1 + s_2 + \displaystyle \frac{\delta-f_2}{1-f_2}s_3 + s_4 & \geq 1 &
(T_2) \\[0.1in]
&  s  \in  \mathbb{R}_+^4 .
\end{array}
\end{equation}

Using the equation of the edge of $Q$ connecting
$f + r^2$ and $f + r^3$, we can find bounds on
$\psi_{T_1}(r^2)$ and $\psi_{T_2}(r^3)$. The edge has equation $x_1 \frac{1}{t}
+ \frac{t-1}{t} x_2 = 1$ for some $1 < t < \infty$.
Therefore $\alpha \leq t$ and $\delta \leq \frac{t}{t-1}$.
Using these two inequalities together with $f_1 \leq \frac{1}{2}$ and $f_2 \leq \frac{1}{2}$
we get
\[
\frac{\alpha-f_1}{1-f_1} = \frac{\alpha - 1}{1-f_1} + 1 \leq 2(t-1)
+ 1 = 2t -1
\hspace{1cm} \mbox{and} \hspace{1cm}
\frac{\delta-f_2}{1-f_2} \leq 2\frac{t}{t-1} - 1 \ .
\]

Using these bounds, we obtain the relaxation of LP (\ref{LP8})

\begin{equation}
\begin{array}{rlcl}
  \;\; \min & s_1 + s_2 + s_3 + s_4 \\[0.1in]
         &       s_1 + (2t-1)s_2 + s_3 + s_4 & \geq 1 &
(T_1) \\[0.1in]
          &  s_1 + s_2 + \displaystyle (2\frac{t}{t-1} - 1)s_3 + s_4 & \geq 1 &
(T_2) \\[0.1in]
&  s  \in  \mathbb{R}_+^4 .
\end{array}
\end{equation}

Set $\lambda = 2t - 1$ and $\mu = 2\frac{t}{t-1} - 1$. Then $t > 1$
implies $\lambda >1$ and $\mu > 1$. The optimal solution of the
above LP is $s_1 = s_4 = 0$, $s_2 = \frac{\mu - 1}{\lambda\mu - 1}$
and $s_3 = \frac{\lambda - 1}{\lambda\mu - 1}$ with value

\[
s_1 + s_2 + s_3 + s_4 = \frac{\lambda + \mu - 2}{\lambda\mu - 1} = \frac{t^2 - 2t
+ 2}{t^2} \ .
\]

To find the minimum of this expression for $t >1$, we set its derivative to 0,
and get the solution $t=2$. Thus the minimum value of
$s_1 + s_2 + s_3 + s_4$ is equal to $\frac{1}{2}$.

\end{proof}

\subsection{Approximating the integer hull by the quadrilateral
closure}
\label{sub:IHbyQuad}

In this section, we study Problem (\ref{eq:facet_over_Qf}).
We can approximate the
facets derived from Type $1$ and Type $2$ triangles using
quadrilaterals in a similar manner as the splits were approximated
by triangles and quadrilaterals. See Figure~\ref{fig:triangle_with
quad}. We again define the set $X$ of points which lie strictly
inside the Type $1$ or Type $2$ triangle, similar to the proof
of Theorem~\ref{THM:inclusion}. Then we can find quadrilaterals as shown in
Figure~\ref{fig:triangle_with quad} that contain the set $X$. The
proof goes through in exactly the same manner.

However triangles of Type $3$ pose a problem. They
cannot be approximated to any desired precision by a
sequence of quadrilaterals.

\begin{figure}[htbp]
\begin{center}
\scalebox{.5}{\epsfig{file=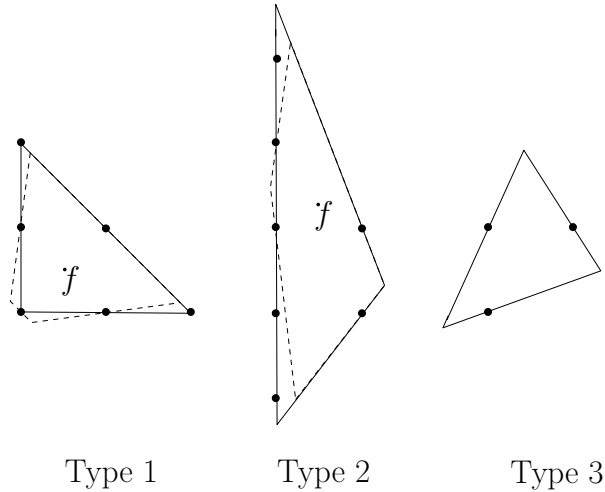}}
\caption{Approximating a triangle inequality with a
quadrilateral inequality} \label{fig:triangle_with quad}
\end{center}
\end{figure}

In this section, we work under Assumption \ref{ass:scaling}.
By Theorem \ref{THM:mainRfk}, a Type $3$ triangle $T$ defines a facet
if and only if either the ray condition holds, or
all three corner rays are present. First we consider the case where the ray condition holds.

\begin{theorem} \label{thm:7.2}
Let $T$ be a triangle of Type 3 with corresponding minimal function
$\psi_T$ defining a facet of $R_f^k$. If the ray condition holds for
$T$ then Problem (\ref{eq:facet_over_Qf}) has optimal value 1.
\end{theorem}

\begin{proof}
We first prove that if the ray condition holds the points
$p^j = f + r^j$ are integral points on the boundary of $T$,
for $j = 1, \ldots, k$.

For $i=1, 2, 3$, let $P_i$ be the set of points left at the end of Step $i$
of the Reduction Algorithm given in Section \ref{SEC:PRELIM}.
Each $p^j \in P_1$ with $p^j$
integral is uniquely active and is removed during Step 2 of
the Reduction Algorithm. Hence, all points in $P_2$ are non-integral.
Observe that Step 3 can only remove boundary points $p$ and $q$
when the segment
$pq$ contains at least two integral points in its relative interior.
Therefore this step does not remove anything in a Type 3 triangle
and $P_3 = P_2$.

Since the ray condition holds, we have $P_3 = P_2 = \emptyset$ and
$P_1$ contains only integral points. But then $P_1 = P$,
showing that all boundary points
at the beginning of the Reduction Algorithm are integral.

It is then straightforward to construct a maximal lattice-free
quadrilateral $Q$ with $p^j$, $j = 1, \ldots, k$ on its boundary,
and containing $f$ in its interior.
It follows that
the value of Problem (\ref{eq:facet_over_Qf}) is equal to 1.
\end{proof}

We now consider the case where $T$ is a Type $3$ triangle with the
three corner rays present. In this case, we can approximate the
facet obtained from $T$ to within a factor of two by using
inequalities derived from triangles of Type $2$. Define another
relaxation $\bar{T}_f^k$ as the convex set defined by the
intersection of the inequalities derived only from Type $1$ and Type
$2$ triangles. By definition, $T_f^k\subseteq \bar{T}_f^k$. From the
discussion at the beginning of this section, we also know
$Q_f^k\subseteq \bar{T}_f^k$. Hence (\ref{eq:facet_over_Qf}) can be
relaxed to
\begin{equation}\label{eq:facet_over_barTf}
\inf\left\{ \displaystyle \sum_{i=1}^k \psi(r^i) s_i: \; s \in
\bar{T}_f^k\right\} \ .
\end{equation}

\begin{theorem}\label{thm:type3_with_type2}
Let $T$ be a triangle of Type 3 with corresponding minimal function
$\psi$ and generating a facet $\displaystyle \sum_{i=1}^{k}
\psi(r^i) s_i \geq 1$ of $R_f^k$. Then,
$$\inf\left\{ \sum_{i=1}^k \psi(r^i) s_i \ : \; s \in
\bar{T}_f^k\right\} \geq \frac{1}{2} \ .
$$
\end{theorem}

This theorem implies directly the following corollary.

\begin{corollary}
$Q_f^k \subseteq 2R_f^k$.
\end{corollary}

\begin{proof}[Proof of Theorem~\ref{thm:type3_with_type2}]
We first make an affine transformation to simplify computations.
Let $y^1, y^2, y^3$ be the three lattice points on the
sides of $T$. We choose the transformation such that the two following
properties are satisfied.

\begin{itemize}
\item[(i)] The standard lattice is mapped to the lattice
generated by the vectors $v^1 = (1,0)$ and
$v^2 = (\frac{1}{2}, \frac{\sqrt{3}}{2})$, i.e. all points of the form
$z_1 v^1 + z_2 v^2$, where $z_1, z_2$ are integers.
\item[(ii)]
$y^1, y^2, y^3$ are respectively mapped to $(0,0),(1,0),(0,1)$ in the above
lattice.
\end{itemize}

We use the basis $v^1, v^2$ for
$\mathbb{R}^2$ for all calculations and equations in the remainder of the
proof. See Figure~\ref{fig:triangle_with triangle}.

\begin{figure}[htbp]
\begin{center}
\scalebox{.5}{\epsfig{file=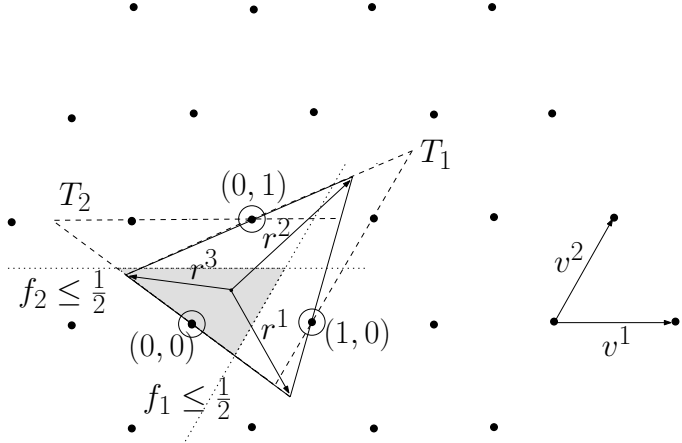}}
\caption{Approximating a Type $3$ triangle inequality with Type $2$ triangle
inequalities. The Type 3 triangle is in solid lines. The basis vectors are
$v^1 = (1, 0)$ and $v^2 = (\frac{1}{2}, \frac{\sqrt{3}}{2})$
}
\label{fig:triangle_with triangle}
\end{center}
\end{figure}

With this transformation, we can get a simple characterization for
the Type $3$ triangles. See Figure~\ref{fig:triangle_with triangle}
for an example. We make the following claim about the relative
orientations of the three sides of the Type $3$ triangle.

\begin{lemma} \label{LEM:type_3_triangle}
Let $\mathcal{F}$ be the family of triangles formed by three lines
given by:
\[
\begin{array}{rcc}
\textrm{Line }1 :
& \displaystyle -\frac{x_1}{t_1} + x_2 = 1 & \mbox{with \ }
0 < t_1 < \infty \ ; \\[0.1in]
\textrm{Line }2 :
& t_2 \ x_1 + x_2 = 0 & \mbox{with \ } 0< t_2 < 1 \ ; \\[0.1in]
\textrm{Line }3 :
& x_1 + \displaystyle \frac{x_2}{t_3} = 1 & \mbox{with \ }
1 < t_3 < \infty \ .
\end{array}
\]

Any Type $3$ triangle is either a triangle from $\mathcal{F}$
or a reflection of a triangle from $\mathcal{F}$ about the
line $x_1 = x_2$.
\end{lemma}

\begin{proof}

Take any Type $3$ triangle $T$. Consider the edge passing through
$(1,0)$. Since it cannot go through the interior of the equilateral
triangle $(0,0), (0,1), (1,0)$, there are only two choices for its
orientation : a) It can go through the segment $(0,1),(1,1)$, or b)
It can go through the segment $(0,0),(1,-1)$.

In the first case its equation is that of $\textrm{Line }3$. This
now forces the edge of $T$ passing through $(0,0)$ to have the
equation for $\textrm{Line }2$. This is because the only other
possibility for this edge would be for the line to pass through the
segment $(-1,1),(0,1)$. But then the lattice point $(1,-1)$ is
included in the interior of the triangle. Similarly, the third edge
must now have $\textrm{Line }1$'s equation, because $(-1,1)$ needs
to be excluded from the interior.

Case b) can be mapped to Case a) by a reflection about the line $x_1 = x_2$.
\end{proof}

\begin{remark}
We can choose any values for $t_1, t_2, t_3$ independently in
the prescribed ranges, and we get a lattice free triangle. This
observation shows that the family $\mathcal{F}$ defined above is
exactly the family of all Type $3$ triangles modulo an affine
transformation.
\end{remark}

We now show how to bound Problem (\ref{eq:facet_over_barTf}) and
hence prove Theorem~\ref{THM:T&Q_2}.

Consider any Type $3$ triangle $T$. It
is sufficient to consider the case where the lines supporting the edges
of $T$ have equations as stated in the Lemma \ref{LEM:type_3_triangle}.

We consider two cases for the position of the fractional point
$f = (f_1,f_2)$.
\begin{enumerate}
\item[(i)] $f_1\leq\frac{1}{2}$, $f_2\leq\frac{1}{2}$;
\item[(ii)] $f_1 \leq 0$, $f_1+f_2\leq \frac{1}{2}$.
\end{enumerate}

The union of the two regions described above when rotated by
$2\pi/3$ and $4\pi/3$ about the point $(\frac{1}{3},\frac{1}{3})$
(the centroid of the triangle formed by (0,0),(0,1) and (1,0)),
cover all of $\mathbb{R}^2$. Hence they cover all of $T$ and by
rotational symmetry, investigating these two cases is enough.

For the first case, we relax Problem (\ref{eq:facet_over_barTf}) by
using only two inequalities from $\bar{T}_f^k$. These are derived
from Type $2$ triangles $T_1$ and $T_2$ (see
Figure~\ref{fig:triangle_with triangle}), which are defined as
follows. $T_1$ has $\textrm{Line }1$ and $\textrm{Line }2$
supporting two of its edges and $x_1 = 1$ supporting the third one
(with more than one integral point). $T_2$ has $\textrm{Line }2$ and
$\textrm{Line }3$ supporting two of its edges and $x_2 = 1$
supporting the third one (with more than one integral point). Let
$\psi_{T_1}$ and $\psi_{T_2}$ be the corresponding minimal functions
derived from $T_1$ and $T_2$.

The following LP is a relaxation of Problem
(\ref{eq:facet_over_barTf}).

\begin{equation}\label{LP9}
\begin{array}{rlcl}
  \;\; \min & \displaystyle \sum_{i=1}^k s_i \\[0.1in]
         &      \displaystyle \sum_{i=1}^k \psi_{T_1}(r^i)s_i & \geq 1 &
(\textrm{Triangle }T_1) \\[0.1in]
        & \displaystyle \sum_{i=1}^k \psi_{T_2}(r^i)s_i & \geq 1 &
(\textrm{Triangle }T_2) \\[0.1in]
          &  s  \in  \mathbb{R}_+^k .
\end{array}
\end{equation}

Theorem~\ref{THM:corner} and Remark~\ref{rem:simplify} imply that
LP (\ref{LP9}) is equivalent to

\begin{equation}\label{LP1}
\begin{array}{rlcl}
  \;\; \min & s_1 + s_2 + s_3 \\
         &      \psi_{T_1}(r^1)s_1 + \psi_{T_1}(r^2)s_2 + \psi_{T_1}(r^3)s_3 & \geq 1 &
(\textrm{Triangle }T_1) \\
        & \psi_{T_2}(r^1)s_1 + \psi_{T_2}(r^2)s_2 + \psi_{T_2}(r^3)s_3 & \geq 1 & (\textrm{Triangle }T_2) \\
          &  s  \in  \mathbb{R}_+^3 .
\end{array}
\end{equation}

where $r^1$, $r^2$ and $r^3$ are the three corner rays (see
Figure~\ref{fig:triangle_with triangle}).

We now show that this LP has an optimal value of at least $\frac{1}{2}$.

Note that $\psi_{T_1}(r^2) = \psi_{T_1}(r^3) = 1$. $\psi_{T_1}(r^1)$
needs to be computed. First we compute $r^1$ and $r^2$ in terms of
$t_1, t_2, t_3, f_1$ and $f_2$.

The intersection of $\textrm{Line }2$ and $\textrm{Line }3$ is given
by
$$
\left(\frac{t_3}{t_3 - t_2}, \frac{-t_2t_3}{t_3 - t_2}\right)
\hspace{0.4in} \mbox{and \ thus} \hspace{0.4in}
r^1 = \left(\frac{t_3}{t_3 - t_2}, \frac{-t_2t_3}{t_3 - t_2}\right) - (f_1,f_2) \ .
$$
As $\psi_{T_1}(r^1)=\frac{1}{\gamma}$ where $\gamma$ is such that
$(f_1,f_2) + \gamma r^1$ lies on the line $x_1=1$, we get
$$
\psi_{T_1}(r^1) = \frac{\frac{t_3}{t_3 - t_2} - f_1}{1-f_1} \ .
$$
Similarly, we only need $\psi_{T_2}(r^2)$ as $\psi_{T_2}(r^1) =
\psi_{T_2}(r^3) = 1$. The intersection of $\textrm{Line }1$ and
$\textrm{Line }3$ is
$$
\left(\frac{t_1(t_3 - 1)}{1 + t_1t_3}, \frac{t_3(t_1+ 1)}{1 +
t_1t_3}\right) \hspace{0.4in} \mbox{and \ thus} \hspace{0.4in} r^2 =
\left(\frac{t_1(t_3 - 1)}{1 + t_1t_3}, \frac{t_3(t_1 + 1)}{1 +
t_1t_3}\right)-(f_1,f_2) \ .
$$
Computing the
coefficient like before, we get
$$
\psi_{T_2}(r^2) = \frac{\frac{t_3(t_1 + 1)}{1 + t_1t_3} - f_2}{1 - f_2} \ .
$$
Hence LP (\ref{LP1}) becomes
\begin{equation}\label{LP2}
\begin{array}{rlcl}
  \;\; \min & s_1 + s_2 + s_3 \\
&  \displaystyle
\frac{\frac{t_3}{t_3 - t_2} - f_1}{1-f_1}s_1 + s_2 +s_3     & \geq 1 &
(\textrm{Triangle }T_1) \\[0.1in]
& s_1 +
\displaystyle
\frac{\frac{t_3(t_1 + 1)}{1 + t_1t_3} - f_2}{1 - f_2}s_2 + s_3 &
\geq 1 & (\textrm{Triangle }T_2) \\
&  s  \in  \mathbb{R}_+^3 .
\end{array}
\end{equation}

As
$$\frac{\frac{t_3}{t_3 - t_2} - f_1}{1-f_1} =
\frac{\frac{t_3}{t_3 - t_2} - 1}{1-f_1} + 1 \ ,
$$
the assumptions
$f_1\leq \frac{1}{2}$ and $t_2 < 1$ yield
$$
\frac{\frac{t_3}{t_3 - t_2} - 1}{1-f_1} + 1
\leq 2\frac{t_3}{t_3 - 1} - 1 \ .
$$
Similarly,
$$
\frac{\frac{t_3(t_1 + 1)}{1 + t_1t_3} - f_2}{1 - f_2} =
\frac{\frac{t_3(t_1 + 1)}{1 + t_1t_3} - 1}{1 - f_2} + 1
$$
and the assumption $f_2\leq \frac{1}{2}$ gives
$$
\frac{\frac{t_3(t_1 + 1)}{1 + t_1t_3} - 1}{1 - f_2} + 1 \leq
2\frac{t_3(t_1 + 1)}{1 + t_1t_3} - 1 = 2\frac{t_3 - 1}{1 + t_1t_3} + 1 \ .
$$
Under the assumption $t_1 > 0$, we obtain
$2\frac{t_3 - 1}{1 + t_1t_3} + 1 \leq 2t_3 - 1$.

To get a lower bound on (\ref{LP2}), we thus can relax its constraints to

\begin{equation}\label{LP3}
\begin{array}{rlcl}
  \;\; \min & s_1 + s_2 + s_3 \\[0.1in]
&  (2 \displaystyle \frac{t_3}{t_3 - 1} - 1)s_1 + s_2 +s_3     & \geq 1 &
(\textrm{Triangle }T_1) \\[0.1in]
& s_1 + (2t_3 - 1)s_2 + s_3 & \geq 1 & (\textrm{Triangle }T_2) \\[0.1in]
&  s  \in  \mathbb{R}_+^3 .
\end{array}\end{equation}

Set $\lambda = 2t_3 - 1$ and $\mu = 2\frac{t_3}{t_3-1} - 1$. Then
$t_3 > 1$ implies $\lambda >1$ and $\mu > 1$.
The optimal solution of LP (\ref{LP3}) is
$s_1 = \frac{\lambda - 1}{\lambda\mu - 1}$,
$s_2 = \frac{\mu - 1}{\lambda\mu - 1}$ and
$s_3 = 0$
with value
\[
s_1 + s_2 + s_3 = \frac{\lambda + \mu - 2}{\lambda\mu - 1} = \frac{t_3^2 -
2t_3 + 2}{t_3^2} \ .
\]
To find the minimum over all $t_3 >1$, we set the derivative to 0,
which gives the solution $t_3=2$. Thus the minimum value of
$s_1 + s_2 + s_3$ is $\frac{1}{2}$.
\bigskip

Next we consider $f_1 \leq 0$ and $f_1+f_2 \leq \frac{1}{2}$, the
shaded region in Figure~\ref{fig:triangle_with triangle_2}. We relax
Problem (\ref{eq:facet_over_barTf}) using only two inequalities.
We take $T_1$ as before, but $T_2$ is the triangle formed by the
following three lines : $\textrm{Line }2$ from
Lemma~\ref{LEM:type_3_triangle}, line parallel to $\textrm{Line }1$
from Lemma~\ref{LEM:type_3_triangle} but passing through $(-1,1)$
and the line passing through $(1,0),(0,1)$.

\begin{figure}[htbp]
\begin{center}
\scalebox{.5}{\epsfig{file=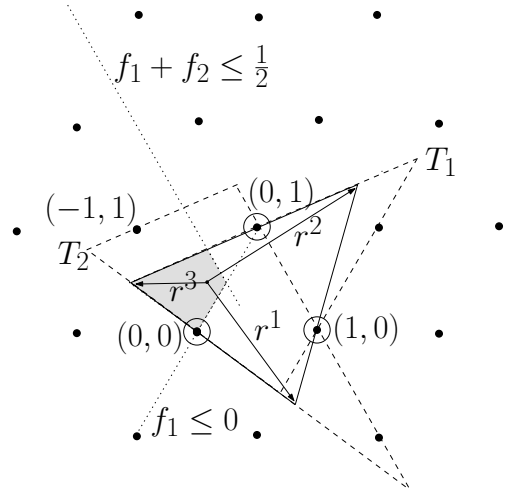}}
\caption{Approximating a Type $3$ triangle inequality with Type $2$
triangles inequalities - Case 2} \label{fig:triangle_with triangle_2}
\end{center}
\end{figure}

As in the previous case, we formulate the relaxation as an LP with
constraints corresponding to $T_1$ and $T_2$. The only difference from
LP (\ref{LP1}) is the coefficient $\psi_{T_2}(r^2)$. This time the
point $(f_1,f_2) + \gamma r^2$ lies on the line $x_1 + x_2 = 1$ (recall
that $\psi_{T_2}(r^2) = \frac{1}{\gamma}$). This gives us
$$
\psi_{T_2}(r^2) = \frac{\frac{2t_1t_3 + t_3 - t_1}{1 + t_1t_3} -f_1-f_2}
{1-f_1-f_2} \ .
$$
We then have the following LP.

\begin{equation}\label{LP4}
\begin{array}{rlcl}
  \;\; \min & s_1 + s_2 + s_3 \\
&  \displaystyle
\frac{\frac{t_3}{t_3 - t_2} - f_1}{1-f_1}s_1 + s_2 +s_3     & \geq 1 &
(\textrm{Triangle }T_1) \\[0.2in]
& s_1 + \displaystyle
\frac{\frac{2t_1t_3 + t_3 - t_1}{1 + t_1t_3} -f_1-f_2}{1-f_1-f_2}s_2 + s_3 & \geq 1 & (\textrm{Triangle }T_2) \\[0.2in]
          &  s  \in  \mathbb{R}_+^3 .
\end{array}
\end{equation}

We simplify the coefficients as earlier :
$$
\frac{\frac{t_3}{t_3 - t_2} - f_1}{1-f_1} =
1 + \frac{\frac{t_3}{t_3 - t_2} - 1}{1-f_1}
\hspace{0.4in} \mbox{and} \hspace{0.4in}
\frac{\frac{2t_1t_3 + t_3 - t_1}{1 + t_1t_3}
-f_1-f_2}{1-f_1-f_2} =
1 + \frac{\frac{2t_1t_3 + t_3 - t_1}{1 + t_1t_3} - 1}{1-f_1-f_2} \ .
$$

Using the assumptions $f_1\leq 0, f_1+f_2\leq\frac{1}{2}$, we
get that
$$
1 + \frac{\frac{t_3}{t_3 - t_2} - 1}{1-f_1} \leq
\frac{t_3}{t_3 - t_2}
\hspace{0.4in} \mbox{and} \hspace{0.4in}
1 + \frac{\frac{2t_1t_3 + t_3 - t_1}{1 + t_1t_3} - 1}{1-f_1-f_2}
\leq 2 (\frac{2t_1t_3 + t_3 - t_1}{1 + t_1t_3}) - 1 \ .
$$
We also have the conditions $t_2 < 1$ and $t_1 > 0$.
$t_2 < 1$ implies $\frac{t_3}{t_3 - t_2} \leq \frac{t_3}{t_3 -
1}$. Moreover
$$
2\left(\frac{2t_1t_3 + t_3 - t_1}{1 + t_1t_3}\right) - 1 =
2\left(2 + \frac{t_3-t_1 -2}{1 + t_1t_3}\right)-1
$$
decreases in value as $t_1$ increases. Its maximum value is less
than the value for $t_1 = 0$, because of the condition $t_1 > 0$.
It follows that $2(\frac{2t_1t_3 + t_3 -
t_1}{1 + t_1t_3}) - 1 \leq 2t_3 - 1$. After putting these
relaxations into the constraints of LP (\ref{LP4}), we get

\begin{equation}\label{LP5}
\begin{array}{rlcl}
\;\; \min & s_1 + s_2 + s_3 \\[0.1in]
&  \displaystyle \frac{t_3}{t_3 - 1}s_1 + s_2 +s_3     & \geq 1 &
(\textrm{Triangle }T_1) \\[0.1in]
& s_1 + (2t_3 - 1)s_2 + s_3 & \geq 1 & (\textrm{Triangle }T_2) \\
&  s  \in  \mathbb{R}_+^3 .
\end{array}
\end{equation}

The optimal solution of LP (\ref{LP5}) is
$$
s_3 = 0, \hspace{0.1in}
s_1 = \frac{2(t_3-1)^2}{2t_3^2 - 2t_3 + 1}, \hspace{0.1in}
s_2 = \frac{1}{2t_3^2 - 2t_3 + 1}, \hspace{0.1in}
\mbox{and} \hspace{0.1in}
s_1 + s_2 + s_3 = \frac{2t_3^2 -4t_3 + 3}{2t_3^2-2t_3+1} \ .
$$
Under the condition $t_3>1$, the minimum value of $s_1 + s_2 + s_3$
is achieved for $t_3=1 + \sqrt{\frac{1}{2}}$ with value $\frac{1}{1
+ \sqrt{\frac{1}{2}}}
> \frac{1}{2}$.
\end{proof}

\section{Split closure vs. a single triangle or quadrilateral
inequality}

In this section, we prove Theorem \ref{THM:T&Q_2}.

This is done by showing that there exist examples of integer
programs (\ref{SI}) where the optimal value for optimizing in the
direction of a triangle (or quadrilateral) inequality over the split
closure $S_f^k$ can be arbitrarily small. We give such examples for
facets derived from triangles of Type $2$ and Type $3$, and from
quadrilaterals.

These examples have the property that the point $f$ lies in the relative
interior of a segment joining two integral points at distance 1.

Furthermore, in these examples, the
rays end on the boundary of the triangle or quadrilateral and hence
the facet corresponding to it is of the form $\sum_{j=1}^k s_j \geq
1$. We show that the following LP has optimal value much less
than 1.

\begin{equation} \label{Split}
\begin{array}{rlcl}
  z_{SPLIT} =    \;\; \min & \displaystyle \sum_{j=1}^k s_j \\[0.1in]
         &       \displaystyle \sum_{j=1}^{k} \psi(r^j) s_j & \geq 1 &
\mbox{ for all splits } B_\psi \\[0.1in]
          &  s  \in  \mathbb{R}_+^k .
\end{array}
\end{equation}

Theorem~\ref{THM:goemans} then implies
Theorem~\ref{THM:T&Q_2}.

A key step in the proof is a method for constructing a polyhedron
contained in the split closure (Lemma \ref{lem:conv_comb}).
The resulting LP implies an upper bound on $z_{SPLIT}$.
We then give a family of examples showing that this upper bound can be
arbitrarily close to 0.
We start the proof with an easy lemma.

\subsection{An easy lemma}

Refer to Figure
\ref{FIG:LConv} for an illustration of the following lemma.

\begin{figure}[htbp]
\centering
\scalebox{.5}{\epsfig{file=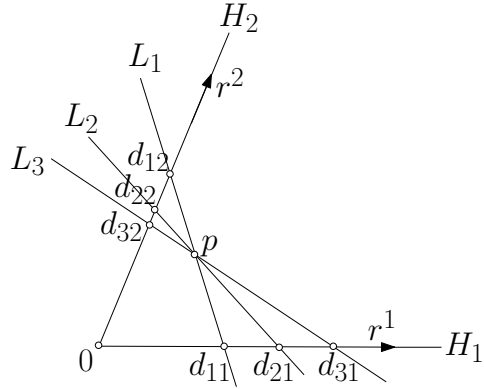}}
\caption{Illustration for Lemma \ref{LEM:LConv}} \label{FIG:LConv}
\end{figure}

\begin{lemma}
\label{LEM:LConv} Let $r^1$ and $r^2$ be two rays that are not
multiples of each others and let $H_1$ and $H_2$ be the half-lines
generated by nonnegative multiples of $r^1$ and $r^2$ respectively.
Let $p := k_1 \ r^1 + k_2 \ r^2$ with $k_1, k_2 > 0$. Let $L_1,
L_2$, and $L_3$ be three distinct lines going through $p$ such that
each of the lines intersects both $H_1$ and $H_2$ at points other
than the origin. Let $d_{ij}$ be the distance from the origin to the
intersection of line $L_i$ with the half-line $H_j$ for $i= 1, 2, 3$
and $j = 1, 2$. Assume that $d_{11} < d_{21} < d_{31}$. Then there
exists $0 < \lambda < 1$ such that
\begin{eqnarray*}
\frac{1}{d_{21}} = \lambda \ \frac{1}{d_{11}} + (1-\lambda) \
\frac{1}{d_{31}} \hspace{1cm} \mbox{and} \hspace{1cm}
\frac{1}{d_{22}} = \lambda \ \frac{1}{d_{12}} + (1-\lambda) \
\frac{1}{d_{32}}.
\end{eqnarray*}

\end{lemma}

\begin{proof}
Let $u^i$ be a unit vector in the direction of $r^i$ for $i=1, 2$.
Using $\{u^1, u^2\}$ as a
base of $\mathbb{R}^2$, for $i=1, 2, 3$, $L_i$
has equation
$$
\frac{1}{d_{i1}} x_1 + \frac{1}{d_{i2}} x_2 = 1
$$
As $L_2$ is a convex combination of $L_1$ and $L_3$, there exists $0
< \lambda < 1$ such that $\lambda L_1 + (1-\lambda) L_3 = L_2$. The
result follows.
\end{proof}

\begin{corollary}\label{COR:LConv}
In the situation of Lemma \ref{LEM:LConv}, let
$L_4$ be a line parallel to $r^1$ going through $p$. Let $d_{42}$ be
the distance between the origin and the intersection of $H_2$ with
$L_4$. Then there exists $0 < \lambda < 1$ such that
\begin{eqnarray*}
\frac{1}{d_{21}} = \lambda \ \frac{1}{d_{11}} \hspace{1cm}
\mbox{and} \hspace{1cm} \frac{1}{d_{22}} = \lambda \
\frac{1}{d_{12}} + (1-\lambda) \ \frac{1}{d_{42}}.
\end{eqnarray*}

\end{corollary}

\begin{proof}
Similar to the proof of Lemma \ref{LEM:LConv}.
\end{proof}

\begin{figure}[htbp]
\begin{center}
\scalebox{.5}{\epsfig{file=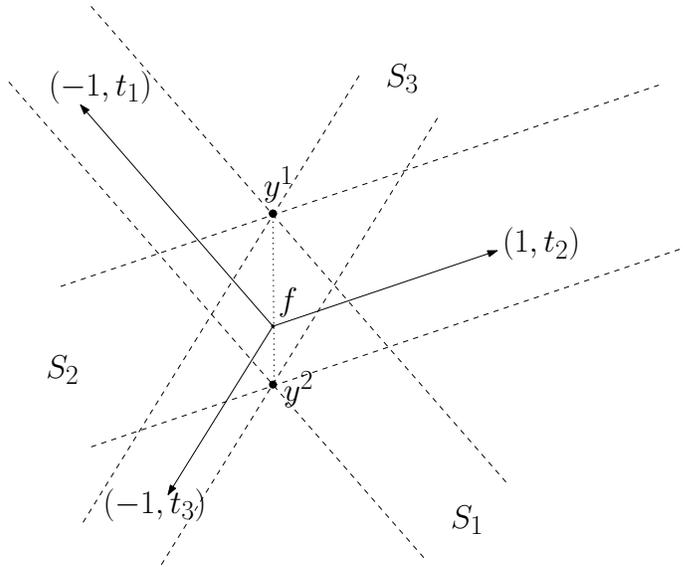}}
\caption{Dominating the Split closure with pseudo-splits}
\label{fig:split_closure}
\end{center}
\end{figure}

\subsection{A polyhedron contained in the split closure}
\label{Subsec:on_segment}

 Our examples for proving Theorem~\ref{THM:T&Q_2} have the property that
the point $f$ lies in the relative interior of a segment joining two
integral points $y^1$, $y^2$ at distance 1.

To obtain an upper bound on the value $z_{SPLIT}$ of the split
closure, we define some inequalities which dominate the split
closure (\ref{Split}). A \emph{pseudo-split} is the convex set
between two distinct parallel lines passing through $y^1$ and $y^2$
respectively. The direction of the lines, called {\it direction} of
the pseudo-split, is a parameter. Figure~\ref{fig:split_closure}
illustrates three pseudo-splits in the directions of three rays. The
\emph{pseudo-split inequality} is derived from a pseudo-split
exactly in the same way as from any maximal lattice-free convex set.
Note that pseudo-splits are in general not lattice-free and hence do
not generate valid inequalities for $R_f^k$. However, we can
dominate any split inequality cutting $f$ by an inequality derived
from these convex sets. Indeed, consider any split $S$ containing
the fractional point $f$ in its interior. Since $f$ lies on the
segment $y^1y^2$, both boundary lines of $S$ pass through the
segment $y^1y^2$. The pseudo-split with direction identical to the
direction of $S$ generates an inequality that dominates the split
inequality derived from $S$, as the coefficient for any ray is
smaller in the pseudo-split inequality.

The next lemma states that we can dominate the split closure by
using only the inequalities generated by the pseudo-splits with
direction parallel to the rays $r^1, \ldots , r^k$ under mild assumptions
on the rays.

\begin{lemma}\label{lem:conv_comb}
Assume that none of the rays $r^1, \ldots, r^k$ has a zero first
component and that $f = (0, f_2)$ with $0 < f_2 < 1$. Let
$y^1 = (0, 1)$ and $y^2 = (0, 0)$, these two points being used to
construct pseudo-splits. Let $S_1, \ldots ,S_k$ be the pseudo-splits
in the directions of rays $r^1, \ldots, r^k$ and denote the
corresponding minimal functions by $\psi_{S_1},\ldots,\psi_{S_k}$.
Let $S$ be any split with $f$ in its interior and let $S'$ be the
corresponding pseudo-split. Then the inequality $\sum_{j=1}^k
\psi_{S'}(r^j) s_j \geq 1$ corresponding to $S'$ is dominated by a
convex combination of the inequalities $\sum_{j=1}^k
\psi_{S_i}(r^j)s_j \geq 1$, $i=1,\ldots, k$. Therefore, the split
inequality corresponding to $S$ is dominated by a convex combination
of the inequalities corresponding to $\psi_{S_1},\ldots,\psi_{S_k}$.
\end{lemma}

\begin{proof}
As a convention, the direction
of a pseudo-split forms an angle with the $x_1$-axis in the range of
$]-\frac{\pi}{2},\frac{\pi}{2}[$. Without loss of generality, assume
that the slope of the directions of the pseudo-splits corresponding to
the rays $r^1, \ldots r^k$ are monotonically non increasing.
We can assume that the direction of $S'$ is different than the direction of
any of the rays in $\{r^1, \ldots, r^k\}$  as otherwise the result
trivially holds.

\begin{figure}[htbp]
\begin{center}
\scalebox{.55}{\epsfig{file=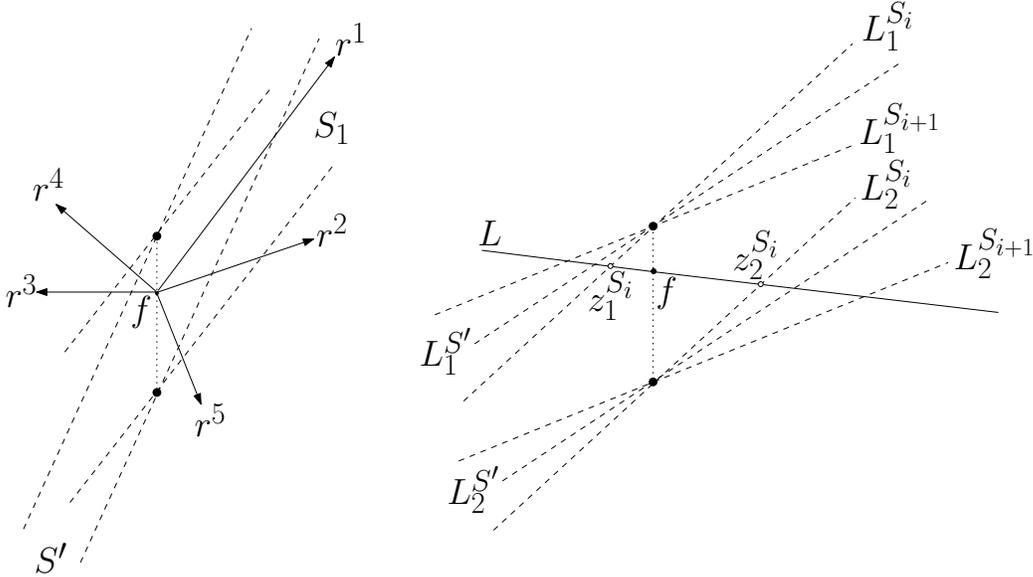}} \caption{Bounding
the split closure with a finite number of pseudo-splits}
\label{fig:split_closure_3}
\end{center}
\end{figure}

First note that, if $S'$ has a direction with slope greater
than the slope of $r^1$, then the inequality
generated by $S'$ is dominated by the one generated by $S_1$.
Indeed, any ray $r^j$ having a slope smaller than $r^1$ has its boundary
point for $S'$ closer to $f$ than the one for $S_1$.
It follows that $\psi_{S'}(r^j) \geq \psi_{S_1}(r^j)$. See
Figure~\ref{fig:split_closure_3}. A similar reasoning holds for the
case where $S'$ has a direction with slope smaller
than the slope of $r^k$.

Thus we only have to consider the case where the slope of the direction of
$S'$ is strictly between the
slopes of the directions of $S_i$ and $S_{i+1}$, for some $1 \le i \le k-1$.
We claim the following.

\begin{observation}\label{obs:conv_comb}
There exists a $0< \lambda <1$ such that $\psi_{S'}(r) =
\lambda\psi_{S_i}(r) + (1-\lambda)\psi_{S_{i+1}}(r)$ for every ray
$r \in \{r^1, \ldots, r^k\}$.
\end{observation}

\begin{proof}
For each pseudo-split $S \in \{S', S_i, S_{i+1}\}$, we denote by
$L^S_1$ its boundary line passing though $(0,1)$ and by $L^S_2$
its boundary line passing through $(0,0)$.

Consider first any ray $r^j$ with $j < i$ and let
$L^{r^j}$ be the half-line $f + \mu r^j$, $\mu \ge 0$.
We have that $L^{r^j}$ has a slope greater than the slope of the
direction of $S_i$ and thus $L^{r^j}$ intersects the boundaries of $S', S_i$
and $S_{i+1}$ on $L^{S'}_1$, $L^{S_{i+1}}_1$ and $L^{S_i}_1$.
By Lemma~\ref{LEM:LConv},
there exists a $0 < \lambda_1 < 1$ such
that, for all $r \in \{r^1, \ldots, r^{i-1}\}$
\begin{equation}\label{eq:conv_comb1}
\psi_{S'}(r) = \lambda_1\psi_{S_i}(r) + (1-\lambda_1)\psi_{S_{i+1}}(r) \ .
\end{equation}

By Corollary~\ref{COR:LConv}, equation (\ref{eq:conv_comb1}) also holds
for $r = r^i$.

Using a similar reasoning for the rays $\{r^{i+1}, \ldots, r^k\}$
and the boundary lines $L^{S'}_2$, $L^{S_{i+1}}_2$ and $L^{S_i}_2$,
there exists a $0 < \lambda_2 < 1$ such that, for all
$r \in \{r^{i+1}, \ldots, r^{k}\}$
\begin{equation}\label{eq:conv_comb2}
\psi_{S'}(r) = \lambda_2\psi_{S_i}(r) + (1-\lambda_2)\psi_{S_{i+1}}(r) \ .
\end{equation}

It remains to show that $\lambda_1 = \lambda_2$. Consider any line $L$
through $f$ that is not collinear with $r^i$ or $r^{i+1}$ and not going
through $y^1$ and $y^2$. For each $S \in
\{S', S_i, S_{i+1}\}$, let $z^S_1$ (resp. $z^S_2$) be the
intersection of $L$ with $L^{S}_1$ (resp. $L^S_2$) and let $d^S_{1}$
(resp. $d^S_{2}$) be the distance from $f$ to $z^S_1$ (resp.
$z^S_2$). See Figure~\ref{fig:split_closure_3}. By
Lemma~\ref{LEM:LConv}

\begin{eqnarray}\label{eq:equalLambdas}
\frac{1}{d^{S'}_{1}} = \lambda_1 \ \frac{1}{d^{S_i}_{1}} + (1-\lambda_1) \
\frac{1}{d^{S_{i+1}}_{1}} \hspace{1cm} \mbox{and} \hspace{1cm}
\frac{1}{d^{S'}_{2}} = \lambda_2 \ \frac{1}{d^{S_i}_{2}} + (1-\lambda_2) \
\frac{1}{d^{S_{i+1}}_{2}}.
\end{eqnarray}

The length of segments $fy^1, fz^S_1, fy^2$ and $fz^S_2$ are
respectively $1-f_2, d^S_1, f_2$ and $d^S_2$.
Observe that the triangles $fz^S_1y^1$ and $fz^S_2y^2$
are homothetic with homothetic ratio $t = \frac{1-f_2}{f_2}$.
It follows that
%
$\frac{d^S_1}{d^S_2} = t$.
%
Substituting  $d^S_1$ by $t \cdot d^S_2$ in (\ref{eq:equalLambdas})
yields $\lambda_1 = \lambda_2$.
\end{proof}

This observation proves the lemma.

\end{proof}

Using the above lemma, we can bound the split closure for three
rays. We assume that none of the rays has a zero first component
and that the three rays generate $\mathbb{R}^2$.
Without loss of generality, we make the
following assumptions. The rays are $r^1 = \mu_1(-1, t_1)$, $r^2 =
\mu_2(1,t_2)$ and $r^3 = \mu_3(-1, t_3)$, where $t_i$'s are rational
numbers in the range $]-\infty,\infty[$, with $t_1 > t_3$ and
$\mu_i$'s are scaling factors with $\mu_i > 0$. Any configuration of
three rays satisfying the above assumptions either fits this
description or is a reflection of it about the segment
$(0,0),(0,1)$. In addition, we must have $-t_1 < t_2 < -t_3$. See
Figure~\ref{fig:split_closure} for an illustration.

\begin{theorem}
Assume that $f = (0, f_2)$ with $0 < f_2 < 1$.
Consider rays $r^1 = \mu_1(-1, t_1)$, $r^2 = \mu_2(1,t_2)$ and $r^3 =
\mu_3(-1, t_3)$, where $t_i$'s are rational numbers with
$-t_1 < t_2 < -t_3$ and $\mu_i > 0$. Then
$$
z_{SPLIT} \leq
\frac{1}{t_1 - t_3}\left(\frac{1-f_2}{\mu_1}+\frac{f_2}{\mu_3}\right) \ .
$$
\end{theorem}

\begin{proof}
Let $y^1 = (0, 1)$ and $y^2 = (0, 0)$, these two points being used to
construct pseudo-splits.
By Lemma~\ref{lem:conv_comb}, we know that the three pseudo-splits
$S_1, S_2, S_3$ corresponding to the directions of $r^1, r^2, r^3$
dominate the entire split closure. More formally, the following LP
is a strengthening of (\ref{Split}) in this example of three rays.

\begin{equation}\label{eq:split_closure}
\begin{array}{rlcl}
   \min & s_1 + s_2 + s_3 & &\\
&\psi_{S_1}(r^1)s_1 + \psi_{S_1}(r^2)s_2 + \psi_{S_1}(r^3)s_3 & \geq 1 & \\
&\psi_{S_2}(r^1)s_1 + \psi_{S_2}(r^2)s_2 + \psi_{S_2}(r^3)s_3 & \geq 1 & \\
&\psi_{S_3}(r^1)s_1 + \psi_{S_3}(r^2)s_2 + \psi_{S_3}(r^3)s_3 & \geq 1 & \\
&  s  \in  \mathbb{R}_+^3 .
\end{array}
\end{equation}

It is fairly straightforward to compute the coefficients in the above
inequalities. We give the calculations for $S_1$; the coefficients for
the other two follow along similar lines.

$\psi_{S_1}(r^1)$ is $0$, since $r^1$ is parallel to the direction of $S_1$.

Consider $r^2$ and let its boundary point $p$ for $S_1$ be
$(0,f_2) + \gamma \mu_2(1,t_2)$, for some $\gamma \geq 0$. Then
$\psi_{S_1}(r^2)$ is $\frac{1}{\gamma}$. To compute $\gamma$, we
observe that $p$ is on boundary $1$ of $S_1$, by assumption of
$t_2 > -t_1$. Hence, the slope of the line connecting $p$ and $(0,1)$ is
$-{t_1}$. Therefore,

\[
\frac{f_2 + \gamma\mu_2 t_2 - 1}{0 + \gamma\mu_2} = -{t_1}
\]

which yields $\gamma = \frac{1-f_2}{\mu_2(t_1 + t_2)}$. Hence
$\psi_{S_1}(r^2) = \frac{\mu_2(t_1 + t_2)}{1-f_2}$.

Now consider $r^3$. As before, let its boundary point $p'$ for $S_1$
be $(0,f_2) + \gamma' \mu_3(-1,t_3)$, for some $\gamma' \geq 0$.
This time note that the ray intersects boundary $2$ (by the assumption
$t_3 < t_1$). Equating slopes, we get
\[
\frac{f_2 + \gamma' \mu_3t_3}{0 - \gamma'\mu_3} = -{t_1}
\]

which yields $\gamma' = \frac{f_2}{\mu_3(t_1 - t_3)}$. Hence
$\psi_{S_1}(r^3) = \frac{\mu_3(t_1 - t_3)}{f_2}$.

So we have that the inequality corresponding to $S_1$ is
\[
0\cdot s_1 + \frac{\mu_2(t_1 + t_2)}{1-f_2}s_2 + \frac{\mu_3(t_1 -
t_3)}{f_2}s_3 \geq 1 \ .
\]

By very similar calculations, we can get the inequalities corresponding to
$\psi_{S_2}$ and
$\psi_{S_3}$. LP (\ref{eq:split_closure}) becomes

\begin{equation}\label{eq:split_closure2}
\begin{array}{rlcl}
   \min & s_1 + s_2 + s_3 & &\\
&    0\cdot s_1 + \displaystyle
\frac{\mu_2(t_1 + t_2)}{1-f_2}s_2 + \frac{\mu_3(t_1 - t_3)}{f_2}s_3 &
\geq 1 & \\[0.1in]
&   \displaystyle \frac{\mu_1(t_1 + t_2)}{1-f_2}s_1 + 0\cdot s_2 +
\frac{\mu_3(-t_3 - t_2)}{f_2}s_3     & \geq 1 & \\[0.1in]
& \displaystyle \frac{\mu_1(t_1 -t_3)}{1-f_2}s_1 +
\frac{\mu_2(-t_3 - t_2)}{f_2}s_2 + 0 \cdot s_3  & \geq 1 & \\[0.1in]
&  s  \in  \mathbb{R}_+^3 .
\end{array}
\end{equation}

As a sanity check, note that the assumption $-t_1 < t_2 < -t_3$
implies that all the coefficients are nonnegative.

The following solution is feasible for LP (\ref{eq:split_closure2}):
$$
s_1 = \frac{1-f_2}{\mu_1(t_1 - t_3)}, \hspace{0.1in}
s_2 = 0, \hspace{0.1in}
s_3 = \frac{f_2}{\mu_3(t_1 -t_3)} \hspace{0.1in}
\hspace{0.1in}  \mbox{and} \hspace{0.1in}
s_1 + s_2 + s_3  =
\frac{1}{t_1 - t_3}\left(\frac{1-f_2}{\mu_1}+\frac{f_2}{\mu_3}\right) \ .
$$
Since the above LP was a strengthening of (\ref{Split}),
we obtain
$$
z_{SPLIT} \leq s_1 + s_2 + s_3 =
\frac{1}{t_1 - t_3}\left(\frac{1-f_2}{\mu_1}+\frac{f_2}{\mu_3}\right) \ .
$$
\end{proof}

If the rays are such that $\mu_1 = \mu_3 = 1$, then the above
expression is $\frac{1}{t_1 - t_3}$. This implies that in this case
if we have rays such that $(t_1 - t_3)$ tends to infinity, then
$z_{SPLIT}$ tends to $0$.

\subsection{Type $2$ triangles that do much better than the split closure}
\label{sec:type2}

In Section \ref{Subsec:on_segment}, we showed that we can bound the
value of the split closure under mild conditions on $f$ and the
rays. In particular, we showed that as $t_1 - t_3$ increases in
value, the split closure does arbitrarily bad. In this section, we
consider an infinite family of Type $2$ triangles with rays pointing
to its corners which satisfy these conditions.

Consider the same situation as in Section \ref{Subsec:on_segment}
and consider the Type $2$ triangle $T$ with the following three
edges. The line parallel to the $x_2$-axis and passing through
$(-1,0)$ supports one of the edges, and the other two edges are
supported by lines passing through $(0,1)$ and $(0,0)$ respectively.
See left part of Figure~\ref{fig:type2_triangle}. Note that in this
example, the rays are of the form $r^1 = (-1,t_1), r^2 = \mu(1,t_2),
r^3 = (-1,t_3)$. In the notation of Section~\ref{Subsec:on_segment},
$\mu_1 = \mu_3 = 1$.

\begin{figure}[htbp]
\begin{center}
\scalebox{.5}{\epsfig{file=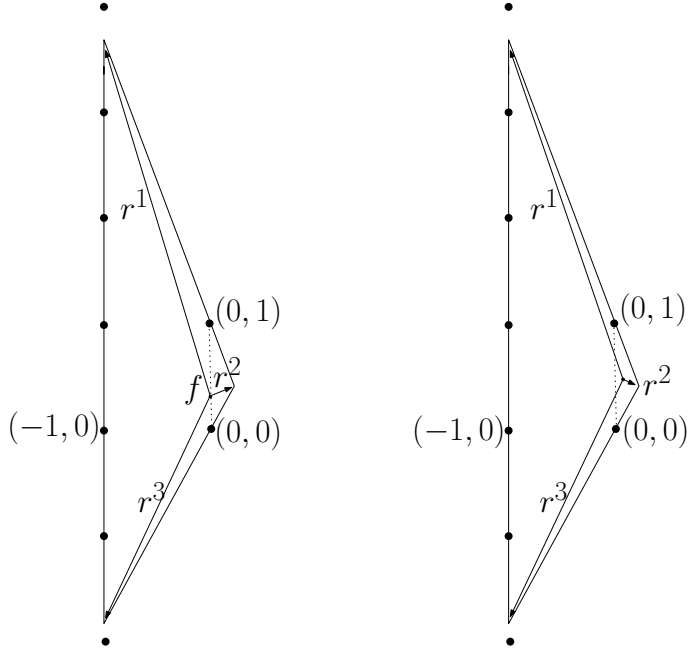}}
\caption{Facets from Type 2 triangles with large gap versus the
split closure}
\label{fig:type2_triangle}
\end{center}
\end{figure}

\begin{theorem}
Given any $\alpha > 1$, there exists a Type 2
triangle $T$ as shown in  Figure~\ref{fig:type2_triangle}
such that for any point $f$ in the relative interior of the
segment joining $(0,0)$ to $(0,1)$, LP
(\ref{Split}) has value $z_{SPLIT} \leq \frac{1}{\alpha}$.
\end{theorem}

\begin{proof}
Let $M = \lceil\alpha\rceil$.
When the fractional point  $f$ is on the segment connecting
$(0,0)$ and $(0,1)$, consider the triangle $T$ with $M$ integral points in the
interior of the vertical edge (the triangle on the left in
Figure~\ref{fig:type2_triangle}). This implies $t_1 - t_3 \geq M$.
Therefore, from the result of Section~\ref{Subsec:on_segment},
$\mu_1 = \mu_3 = 1$ implies that $z_{SPLIT}\leq \frac{1}{t_1 - t_3}
\leq \frac{1}{\alpha}$.

\old{
When the fractional point is not on the
segment connecting $(0,0),(0,1)$ but is in the interior of triangle $\Delta$
with vertices $(0,0)$, $(0,1)$ and the vertex $x^2$ of $T$ with
positive first coordinate (on the right in
Figure~\ref{fig:type2_triangle}), consider a triangle $T$ with $2M$
integral points on the vertical edge. Thus, $t_1 - t_3 \geq 2M$.
From Section~\ref{Subsec:off_segment}, $\mu_1 = \mu_3 = 1$ implies
that $z_{SPLIT}\leq \frac{2}{t_1 + t_3} \leq \frac{1}{M} \leq
\frac{1}{\alpha}$.
}
\end{proof}

In this example, for any large constant $\alpha$, optimizing
over the split closure in the direction of the facet defined by
these Type $2$ triangles yields at most $\frac{1}{\alpha}$. This
implies Theorem~\ref{THM:T&Q_2}.

\subsection{More bad examples}

The examples of Section \ref{sec:type2} can be modified in various ways while keeping the property that the split closure is arbitrarily bad. The proofs are similar to that of Theorem~\ref{THM:T&Q_2}.

\subsubsection{Type 2 triangles when $f$ is not on the segment
joining $(0,0)$ to $(0,1)$}

The example of Section \ref{sec:type2} can be generalized to the case
where $f$ is not on the segment connecting the points $(0,0)$ and $(0,1)$
as follows. Let $T$ be a Type 2 triangle as shown on the right part of
Figure~\ref{fig:type2_triangle}. Let $\Delta$  be the triangle
with vertices $(0,0)$, $(0,1)$ and the vertex $x^2$ of $T$ with
positive first coordinate.
When the fractional point $f$ is in the interior of triangle $\Delta$,
and triangle $T$ has $2M$ integral points on its vertical edge, one can show that $z_{SPLIT}\leq \frac{1}{M}$.

However, such bad examples cannot be constructed for any position of
point $f$ in the triangle $T$. In particular, define the triangle
$\Delta '$ obtained from $\Delta$ by a homothetic transformation
with center $x^2$ and factor 2 (so one vertex of $\Delta '$
is $x^2$ and points $(0,0)$ and $(0,1)$ become the middle points of the
two edges of $\Delta '$ with endpoint $x^2$).
When $f$ is an interior point of $T$ outside $\Delta '$, it is easy
to see that the split inequality obtained from the split parallel
to the $x_2$-axis $-1 \leq x_1 \leq 0$ approximates the triangle
inequality defined by $T$ to within a factor at most 2. Indeed the
linear program is

\begin{equation}
\begin{array}{rlcl}
\min & s_1 + s_2 + s_3 & &\\[0.1in]
 & \displaystyle
 s_1 + \frac{f_1 - u}{f_1}s_2 + s_3 & \geq 1 & \\[0.1in]
&  s  \in  \mathbb{R}_+^3 ,
\end{array}
\end{equation}

\noindent where $u$ is the first coordinate of $x^2$. The optimal
solution is $s_1 = 0$, $s_2=\frac{f_1}{f_1-u}$, $s_3=0$. Thus
$s_1+s_2+s_3= \frac{f_1}{f_1-u} \geq \frac{1}{2}$ since $f_1 \leq
-u$ for any $f \in T \setminus \Delta '$. This implies that the
split inequality approximates the triangle inequality by a factor at
most 2 when $f$ is outside $\Delta '$.

\begin{figure}[htbp]
\begin{center}
\scalebox{.5}{\epsfig{file=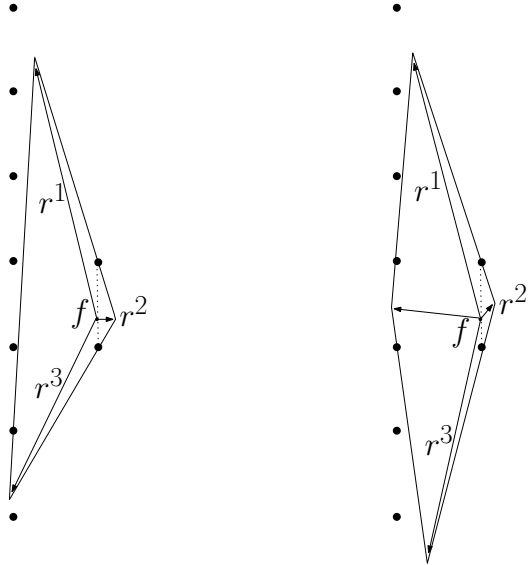}} \caption{Facets from
Type 3 triangles and quadrilaterals on which the split closure does
poorly} \label{fig:type2_and_quad}
\end{center}
\end{figure}
\subsubsection{Triangles of Type 3 and quadrilaterals}

We now show how to modify the construction of Section
\ref{sec:type2} to get examples of Type $3$ triangles and
quadrilaterals that do arbitrarily better than the split closure.

To get a Type 3 triangle, we tilt the vertical edge of the triangle in
Figure~\ref{fig:type2_triangle} around its integral point with minimum
$x_2$-value. See Figure \ref{fig:type2_and_quad}. The same bound on
$z_{SPLIT}$ is then achieved.

Similarly, quadrilaterals can be constructed by breaking the
vertical edge in Figure~\ref{fig:type2_triangle} into two edges of
the quadrilateral. See Figure \ref{fig:type2_and_quad}. By very
similar arguments as in the previous section, we can show that
$z_{SPLIT}$ tends to $0$.

\bigskip
{\bf Acknowledgements:} We thank the referees for their very helpful comments.


\end{document}